\documentclass[11pt]{amsart}

\usepackage[a4paper,margin=2.8cm]{geometry}
\usepackage{microtype}
\usepackage{xcolor}

\usepackage{amsmath,amssymb,amsthm,mathtools}
\usepackage{mathabx}
\usepackage{mathrsfs}
\usepackage[shortlabels]{enumitem}

\usepackage[backref=page]{hyperref}
\hypersetup{
	colorlinks=true,
	linkcolor=blue,
	citecolor=blue,
	urlcolor=blue,
}

\numberwithin{equation}{section}
\allowdisplaybreaks[3]

\newtheorem{theorem}{Theorem}[section]
\newtheorem{proposition}[theorem]{Proposition}
\newtheorem{lemma}[theorem]{Lemma}
\newtheorem{corollary}[theorem]{Corollary}

\theoremstyle{plain}
\newtheorem{definition}[theorem]{Definition}

\theoremstyle{remark}
\newtheorem{remark}[theorem]{Remark}

\newcommand{\D}{\mathfrak{D}}
\newcommand{\Esn}{\mathcal{E}_{s,n,\iota}}

\newcommand{\N}{\mathbb{N}}
\newcommand{\Qmc}{\mathcal{Q}}
\newcommand{\R}{\mathbb{R}}
\newcommand{\Rd}{\mathbb{R}^{d}}
\newcommand{\Sd}{\mathbb{S}^{d-1}}
\newcommand{\Tmc}{\mathcal{T}}
\newcommand{\Z}{\mathbb{Z}}

\newcommand{\BMO}{\operatorname{BMO}}

\newcommand{\Tbmax}{T^{*}_{\Omega,b}}
\newcommand{\Tmax}{T^{*}_{\Omega}}

\newcommand{\card}{\operatorname{card}}
\newcommand{\diam}{\operatorname{diam}}
\newcommand{\dist}{\operatorname{dist}}
\newcommand{\pv}{\operatorname*{p.v.}}
\newcommand{\supp}{\operatorname{supp}}

\newcommand{\eps}{\varepsilon}
\newcommand{\PsiComp}{\overline{\Psi}}
\newcommand{\Bnorm}{\norm{b}_{\BMO}}

\newcommand{\dd}{\,\mathrm{d}}

\newcommand{\abs}[1]{\left\lvert #1\right\rvert}
\newcommand{\brak}[1]{\left[#1\right]}
\newcommand{\inner}[2]{\langle #1,#2 \rangle}
\newcommand{\norm}[1]{\left\lVert #1\right\rVert}
\newcommand{\paren}[1]{\left(#1\right)}
\newcommand{\set}[1]{\left\{#1\right\}}

\title[Weak Endpoint Estimate for   Commutators...]{Weak Endpoint Estimate for   Commutators of Rough Maximal Singular Integral operators}

\author{Shenglong Lin}
\address{Shenglong Lin: 
	School of Mathematical Sciences, Beijing Normal University, Laboratory of Mathematics and Complex Systems, Ministry of Education, Beijing 100875, People’s Republic of China}
\email{linshenglong@mail.bnu.edu.cn}

\author{Qingying Xue$^{*}$}
\address{Qingying Xue:
	School of Mathematical Sciences, Beijing Normal University, Laboratory of Mathematics and Complex Systems, Ministry of Education, Beijing 100875, People’s Republic of China}
\email{qyxue@bnu.edu.cn}

\thanks{The authors acknowledge financial support from the National Key R\&D Program of China (No. 2020YFA0712900) and NNSF of China (No. 12671114)}

\thanks{$^{*}$ Corresponding author: Qingying Xue. }
\date{\today}


\begin{document}
	
	\begin{abstract}
		Let $d\geq 2$ and $T^*_{\Omega,b}$ be the commutator  of the rough maximal singular integral on
		$\mathbb{R}^d$ defined by
	$$
		T^*_{\Omega,b}f(x)
		=
		\sup_{\varepsilon>0}
		\left|
		\int_{|x-y|>\varepsilon}
		\bigl(b(x)-b(y)\bigr)
		\frac{\Omega(x-y)}{|x-y|^d}f(y)\,dy
		\right|.$$
		Assume that $\Omega\in L(\log L)^2(\mathbb{S}^{d-1})$ has mean zero and
		that $b\in\operatorname{BMO}(\mathbb{R}^d)$. We prove that, for every
		$\lambda>0$,
		\[
		\bigl|\{x\in\mathbb{R}^d:T^*_{\Omega,b}f(x)>\lambda\}\bigr|
		\lesssim_{\Omega,b}
		\int_{\mathbb{R}^d}
		\frac{|f(x)|}{\lambda}
		\log\!\left(e+\frac{|f(x)|}{\lambda}\right)\,dx.
		\]
		The argument starts from a Calder\'on-Zygmund decomposition of $f$ and a
		dyadic linearization of the maximal truncation. A decomposition of $\Omega$
		by size, followed by regularization of the spatial cutoffs and a microlocal
		decomposition, reduces the problem to a family of estimates with quantitative
		decay. The required decay follows by interpolating localized $L^1$ and $L^3$
		bounds and using an analytic-family argument for the commutators.
	\end{abstract}
	
\maketitle
\tableofcontents
	
\section{Introduction}\label{sec:introduction}

The main purpose of this paper is to establish an endpoint estimate for
the commutators of maximal singular integral operators under suitable integrability assumptions on $\Omega$ and for
$b\in\operatorname{BMO}(\mathbb{R}^d)$.  We begin with the background of some related operators.

\subsection{Background}

Let $d\geq 2$. Let $\Omega$ be a measurable function on
$\mathbb{R}^{d}\setminus\{0\}$ that is homogeneous of degree zero and has mean
zero on the sphere:
\begin{equation}\label{eq:omega-assumptions}
	\Omega(rx)=\Omega(x)
	\quad (r>0),\qquad
	\int_{\mathbb{S}^{d-1}}\Omega(\theta)\,\dd\sigma(\theta)=0.
\end{equation}
The associated rough homogeneous singular integral is defined by
\begin{equation}\label{eq:rough-singular-integral}
	T_{\Omega}f(x)
	\coloneqq
	\pv\int_{\mathbb{R}^{d}}
	\frac{\Omega(x-y)}{|x-y|^{d}}f(y)\,\dd y.
\end{equation}
The adjective \emph{rough} refers to the absence of regularity assumptions on
the angular part $\Omega$. Strong $L^p$ estimates persist under weak angular
hypotheses, whereas endpoint bounds are far more sensitive to this roughness.

The $L^p$ theory of \eqref{eq:rough-singular-integral} originates in the work
of Calder\'on and Zygmund \cite{CalderonZygmund1956}. Subsequent developments
include the Hardy-space criterion of Connett \cite{C1979}, Ricci and Weiss
\cite{RicciWeiss1979}, the Fourier-transform method of Duoandikoetxea and
Rubio de Francia \cite{DuoandikoetxeaRubio1986}, and the estimates for
maximal truncations obtained by Grafakos and Stefanov
\cite{GrafakosStefanov1998}. The endpoint theory is subtler. Christ's work
\cite{Christ1988} introduced fundamental methods for weak endpoint bounds
for rough operators. Christ and Rubio de Francia
\cite{ChristRubio1988} proved the weak type $(1,1)$ estimate for
$T_\Omega$ under the condition
$\Omega\in L\log L(\mathbb{S}^{d-1})$ for $2\leq d\leq 5$; independently,
Hofmann \cite{Hofmann1988} treated the case $d=2$ under the stronger
assumption $\Omega\in L^q(\mathbb{S}^{1})$, $q>1$. Seeger
\cite{Seeger1996} proved the $L\log L$ theorem in every dimension:
\begin{equation}\label{eq:Seeger-weak11}
	\bigl|\{x\in\mathbb{R}^{d}:|T_{\Omega}f(x)|>\lambda\}\bigr|
	\lesssim_{\Omega}
	\frac{\|f\|_{L^{1}(\mathbb{R}^{d})}}{\lambda},
	\qquad \lambda>0,
\end{equation}
whenever $\Omega\in L\log L(\mathbb{S}^{d-1})$. Tao
\cite{Tao1999} subsequently extended this endpoint principle to homogeneous
convolution operators on homogeneous groups. The condition $L\log L$ is
therefore the natural benchmark for rough singular integrals; see also
\cite{DingLai2019,Bhojak2024} for later refinements and sharpness phenomena.

We next recall the corresponding endpoint theory for commutators. For
$b\in L^1_{\mathrm{loc}}(\mathbb{R}^{d})$, write
\[
	b_Q\coloneqq \frac{1}{|Q|}\int_Q b(x)\,\dd x,
	\qquad
	\|b\|_{\BMO}
	\coloneqq
	\sup_Q\frac{1}{|Q|}\int_Q|b(x)-b_Q|\,\dd x.
\]
Here and below, suprema indexed by $Q$ are taken over cubes in
$\mathbb{R}^{d}$.
The John-Nirenberg inequality \cite{JohnNirenberg1961} gives exponential
integrability of $b-b_Q$ when $b\in\BMO$. For a linear operator $T$, its
commutator with symbol $b$ is
\[
	[b,T]f\coloneqq b\,Tf-T(bf).
\]
Coifman, Rochberg and Weiss \cite{CoifmanRochbergWeiss1976} identified BMO
as the natural symbol class for the strong $L^p$ theory of
Calder\'on-Zygmund commutators. At $p=1$, weak type $(1,1)$ must be
replaced by an Orlicz endpoint. If
\[
	\Phi(t)\coloneqq t\log(e+t),\qquad t\geq 0,
\]
then P\'erez \cite{Perez1995} proved
\begin{equation}\label{eq:Perez-endpoint}
	\bigl|\{x\in\mathbb{R}^{d}:|[b,T]f(x)|>\lambda\}\bigr|
	\lesssim_{T,b}
	\int_{\mathbb{R}^{d}}
	\Phi\left(\frac{|f(x)|}{\lambda}\right)\,\dd x
\end{equation}
for Calder\'on-Zygmund operators $T$ and $b\in\BMO$.

For the rough operator $T_\Omega$, the commutator takes the form
\begin{equation}\label{eq:rough-commutator}
	T_{\Omega,b}f(x)
	\coloneqq [b,T_\Omega]f(x)
	=
	\pv\int_{\mathbb{R}^{d}}
	\bigl(b(x)-b(y)\bigr)
	\frac{\Omega(x-y)}{|x-y|^{d}}f(y)\,\dd y.
\end{equation}
Hu \cite{Hu2003} established strong $L^p$ bounds for
$T_{\Omega,b}$, and for its maximal truncated analogue, under the angular
condition $\Omega\in L(\log L)^2(\mathbb{S}^{d-1})$. At the endpoint,
Lan, Tao and Hu \cite{LanTaoHu2020} first obtained an $L\log L$ estimate
under an $L^q$ assumption on $\Omega$. Hu and Tao \cite{HuTao2023} reached
the rougher class $L(\log L)^2$ and proved
\begin{equation}\label{eq:HuTao-endpoint}
	\bigl|\{x\in\mathbb{R}^{d}:|T_{\Omega,b}f(x)|>\lambda\}\bigr|
	\lesssim_{\Omega,b}
	\int_{\mathbb{R}^{d}}
	\Phi\left(\frac{|f(x)|}{\lambda}\right)\,\dd x.
\end{equation}
Thus the additional logarithm in the angular hypothesis compensates for the
interaction between the rough kernel and the oscillation of the BMO symbol.

Now, we recall some background for the associated maximal singular integral operator.
For $\varepsilon>0$, define
\[
	T_{\Omega,\varepsilon}f(x)
	\coloneqq
	\int_{|x-y|>\varepsilon}
	\frac{\Omega(x-y)}{|x-y|^{d}}f(y)\,\dd y,
	\qquad
	\Tmax f(x)
	\coloneqq
	\sup_{\varepsilon>0}|T_{\Omega,\varepsilon}f(x)|.
\]
The passage from $T_\Omega$ to $\Tmax$ is not formal. The supremum over
$\varepsilon$ destroys linearity, and estimates for the principal value do
not by themselves control the entire family of truncations. This difficulty
also appears in jump and variation inequalities
\cite{DingHongLiu2017,ChenDingHongLiu2021}.

Endpoint estimates for $\Tmax$ remained open even for bounded angular
functions. Di Plinio, Hyt\"onen and Li \cite{DiPlinioHytonenLi2020}
obtained sparse bounds for bounded $\Omega$, while Bhojak and Mohanty
\cite{BhojakMohanty2023} proved an $L\log\log L$-type estimate when
$\Omega\in L\log L$. Hu and Tao \cite{HuTao2024} later established a
bilinear sparse bound with $L\log\log L$ averages. Lai
\cite{Lai2025} recently resolved the weak endpoint problem:
\begin{equation}\label{eq:Lai-maximal-weak11}
	\bigl|\{x\in\mathbb{R}^{d}:T^*_{\Omega} f(x)>\lambda\}\bigr|
	\lesssim_{\Omega}
	\frac{\|f\|_{L^{1}(\mathbb{R}^{d})}}{\lambda},
	\qquad \lambda>0,
\end{equation}
provided $\Omega\in L\log L(\mathbb{S}^{d-1})$.

The object considered here combines the maximal and commutator phenomena.
For $b\in\BMO$ and $\varepsilon>0$, set
\begin{equation*}
	T_{\Omega,b,\varepsilon}f(x)
	\coloneqq
	b(x)T_{\Omega,\varepsilon}f(x)
	-T_{\Omega,\varepsilon}(bf)(x)
	=
	\int_{|x-y|>\varepsilon}
	\bigl(b(x)-b(y)\bigr)
	\frac{\Omega(x-y)}{|x-y|^{d}}f(y)\,\dd y,
\end{equation*}
and define the maximal commutator
\begin{equation}\label{eq:rough-maximal-commutator}
	\Tbmax f(x)
	\coloneqq
	\sup_{\varepsilon>0}
	|T_{\Omega,b,\varepsilon}f(x)|.
\end{equation}
The definition \eqref{eq:rough-maximal-commutator} is essential:
$\Tbmax$ is the supremum of the truncated commutators, not the formal
commutator $[b,\Tmax]$ with the sublinear operator $\Tmax$.

The endpoint estimates \eqref{eq:HuTao-endpoint} and
\eqref{eq:Lai-maximal-weak11} suggest a precise question. 

\noindent {\bf Question.} Does the
$L\log L$ endpoint bound for the rough commutator persist uniformly over
all truncation parameters, without strengthening the
$L(\log L)^2$ hypothesis on $\Omega$?

\subsection{Main results}

In this paper, we  gives an affirmative
answer to the above question. Set
\begin{equation*}
	\Phi(t) \coloneqq t\bigl(\log(e+t)\bigr),
	\quad t > 0.
\end{equation*} Our main result is stated as follows.

\begin{theorem}\label{thm:main}
	Let $d\geq 2$. Suppose that $\Omega$ is homogeneous of degree zero,
	has mean zero on $\mathbb{S}^{d-1}$, and belongs to
	$L(\log L)^2(\mathbb{S}^{d-1})$. If
	$b\in\BMO(\mathbb{R}^{d})$, then, for every $\lambda>0$,
	\begin{equation}\label{eq:main-endpoint}
		\bigl|\{x\in\mathbb{R}^{d}:\Tbmax f(x)>\lambda\}\bigr|
		\lesssim_{\Omega,b}
		\int_{\mathbb{R}^{d}}
		\Phi\left(\frac{|f(x)|}{\lambda}\right)\,\dd x.
	\end{equation}
\end{theorem}

The angular assumption in Theorem \ref{thm:main} is the same as in the
nonmaximal endpoint theorem \eqref{eq:HuTao-endpoint}. In particular, taking
the supremum over truncations entails no further logarithmic loss in the
integrability of $\Omega$. The theorem lies at the intersection of the
endpoint theories for rough commutators and rough maximal truncations, but
its proof requires more than a direct combination of the two estimates.
\vspace{0.3cm}

\noindent {\bf Sketch and novelty of the proof.} To demonstrate our novelty, 
we now explain a little bit of the main ideas in the proof of Theorem \ref{thm:main}.

We begin by replacing arbitrary truncations with dyadic ones and linearizing
the resulting maximal operator. A Calder\'on-Zygmund decomposition of $f$
then separates the good and bad contributions. The good part is controlled
by $L^2$ estimates. For a bad function supported on a cube $Q$, we use
\[
	b(x)-b(y)
	=
	\bigl(b(x)-b_Q\bigr)-\bigl(b(y)-b_Q\bigr).
\]
The term containing $b(y)-b_Q$ is reduced to Lai's estimate
\eqref{eq:Lai-maximal-weak11}, applied to a modified bad function. The term
with $b(x)-b_Q$ outside the singular integral is the principal new
difficulty.

Let $s$ measure the separation between the scale of a bad cube and the scale
of the kernel. At each $s$, we split
\[
	\Omega
	=
	\Omega\chi_{\{|\Omega|>2^{\eta s}\}}
	+
	\Omega\chi_{\{|\Omega|\leq 2^{\eta s}\}}.
\]
The large part is handled in $L^1$ by combining Orlicz estimates with the
John-Nirenberg inequality. The summability in $s$ is exactly where the
$L(\log L)^2$ hypothesis enters. For the bounded part of the kernel, a
Rademacher-Menshov linearization and Lai's dyadic organization reduce the
argument to an $L^2$ estimate with decay in $s$.

To obtain this decay, we regularize the spatial cutoffs and separate the
boundary, low-frequency, and high-frequency contributions. The first two are
controlled by localized $L^1$ and $L^3$ estimates and interpolation. For the
high-frequency term, a microlocal projection $G_\nu$ yields the identity
\[
	(b-b_Q)u
	=
	G_\nu\bigl[(b-b_Q)u\bigr]
	+[b,G_\nu]u
	+(b-b_Q)(I-G_\nu)u,
\]
where $u$ denotes a localized high-frequency piece. This decomposition is
the main point at which the commutator argument departs from the proof for
$\Tmax$. The principal microlocal term is treated by a $TT^*$ argument, the
middle term by analytic conjugation and weighted interpolation, and the
remainder by localized interpolation. Their combined decay is summable in
$s$ and closes the endpoint estimate.
\vspace{0.3cm}

{The organization of the paper is as follows:}
Section \ref{sec:preliminaries} records the Orlicz and BMO estimates,
exponential-weight bounds, and interpolation lemmas used throughout.
Section \ref{sec:reduction-linearization} reduces Theorem \ref{thm:main} to a
signed $L^2$ estimate by dyadic localization, a Calder\'on-Zygmund
decomposition, and linearization. In Section
\ref{sec:linearized-operator-estimate}, we regularize the spatial cutoffs,
introduce the microlocal decomposition, and prove the boundary and
low-frequency estimates. Sections
\ref{sec:proof-lemma45}-\ref{sec:proof-lemma47} establish the principal
microlocal estimate, the commutator estimate, and the remainder estimate,
respectively.
\vspace{0.3cm}

\noindent {\bf Notation.} The surface
measure on $\Sd$ is denoted by $d\sigma$. The symbols $\N$ and $\Z$ denote
the natural numbers and the integers. For a measurable set $E\subset\Rd$,
$|E|$ is its Lebesgue measure and $\chi_E$ its indicator function. The
letter $C$ denotes a positive constant that may change from line to line and
depends only on fixed structural parameters. Dependence on additional
parameters is indicated by subscripts. We write $A\lesssim B$ when
$A\leq CB$, and $A\approx B$ when both $A\lesssim B$ and $B\lesssim A$.

For $1\leq p\leq\infty$, write
$\|f\|_p=\|f\|_{L^p(\Rd)}$. If $1<p<\infty$, then
$p'=p/(p-1)$. The space $L^{1,\infty}(\Rd)$ denotes weak $L^1$. For a cube
$Q\subset\Rd$, let $\ell(Q)$ and $c_Q$ denote its side length and center.
If $a>0$, then $aQ$ is the cube concentric with $Q$ whose side length is
$a\ell(Q)$. We use
\[
	f_Q\coloneqq \frac{1}{|Q|}\int_Q f(x)\,\dd x,
	\qquad
	|f|_Q\coloneqq \frac{1}{|Q|}\int_Q|f(x)|\,\dd x.
\]
The symbols $\supp f$, $\dist(E,F)$, and $\diam(E)$ have their usual
meanings. Our Fourier-transform conventions are
\[
	\widehat f(\xi)
	\coloneqq
	\int_{\Rd}e^{-i\inner{x}{\xi}}f(x)\,\dd x,
	\qquad
	\widecheck f(x)
	\coloneqq
	\frac{1}{(2\pi)^d}
	\int_{\Rd}e^{i\inner{x}{\xi}}f(\xi)\,\dd\xi.
\]
Finally, $\Phi(t)=t\log(e+t)$ throughout. The local Orlicz norm
$\|f\|_{L\log L,Q}$ and the generalized H\"older inequality are recalled in
Section \ref{sec:preliminaries}.

\section{Preliminaries}\label{sec:preliminaries}

We recall here the notation and auxiliary results used in the proof.  These
comprise basic facts about Orlicz spaces and BMO, an exponential-weight
consequence of the John-Nirenberg theorem, and two interpolation principles.

\subsection{Orlicz spaces and BMO space}\label{subsec:young-orlicz}

The material in this subsection is standard; see, for example,
\cite{Grafakos2014,RaoRen1991}.
\begin{definition}\label{def:young-function}
	A continuous function $\Psi\colon[0,\infty)\to[0,\infty)$ is called a
	Young function if it is convex and increasing, $\Psi(0)=0$, and
	$\Psi(t)\to\infty$ as $t\to\infty$.
\end{definition}

\begin{definition}\label{def:orlicz-norm}
	Let $Q\subset\Rd$ be a cube.  For a measurable function $f$ on $Q$, its
	localized Luxemburg norm is defined by
	\begin{equation*}
		\norm{f}_{\Psi(L),Q}
		\coloneqq \inf\set{
			\lambda>0:
			\frac{1}{\abs{Q}}
			\int_Q\Psi\Big(\frac{\abs{f(x)}}{\lambda}\Big)\,\dd x
			\leq1
		}.
	\end{equation*}
	We denote the corresponding local Orlicz space by $\Psi(L)(Q)$.
\end{definition}

The complementary Young function is the Legendre transform
\begin{equation*}
	\PsiComp(t)
	\coloneqq \sup_{s\geq0}\set{st-\Psi(s)},
	\quad t\geq0.
\end{equation*}
Thus Young's inequality takes the form
\begin{equation*}
	st\leq\Psi(s)+\PsiComp(t),
	\quad s,t\geq0.
\end{equation*}
The generalized inverses of $\Psi$ and $\PsiComp$ satisfy
\begin{equation*}
	t
	\leq
	\Psi^{-1}(t)\PsiComp^{-1}(t)
	\leq2t,
	\quad t>0.
\end{equation*}

We shall repeatedly use the following Orlicz-H\"older inequality.
\begin{proposition}\label{prop:orlicz-holder}
	Let $\Psi$ be a Young function, with complementary function $\PsiComp$.
	For every cube $Q\subset\Rd$ and all measurable functions $f,g$ on $Q$,
	\begin{equation*}
		\frac{1}{\abs{Q}}
		\int_Q\abs{f(x)g(x)}\,\dd x
		\leq
		2\norm{f}_{\Psi(L),Q}\norm{g}_{\PsiComp(L),Q}.
	\end{equation*}
	Moreover,
	\begin{equation*}
		\norm{f}_{\Psi(L),Q}
		\approx
		\sup_{\norm{g}_{\PsiComp(L),Q}\leq1}
		\frac{1}{\abs{Q}}
		\int_Q\abs{f(x)g(x)}\,\dd x,
	\end{equation*}
	with absolute comparison constants.
\end{proposition}

The endpoint estimate is naturally expressed in terms of the Young function
\begin{equation*}
	\Phi(t)=t\log(e+t).
\end{equation*}
A direct calculation shows that its complementary function satisfies
\begin{equation*}
	\overline{\Phi}(t)\approx e^t-1
	\qquad (t\ \text{sufficiently large}).
\end{equation*}
Consequently, the associated localized Luxemburg norms obey
\begin{equation*}
	\frac{1}{\abs{Q}}
	\int_Q\abs{f(x)g(x)}\,\dd x
	\lesssim
	\norm{f}_{L\log L,Q}\norm{g}_{\exp L,Q}.
\end{equation*}

For $1<p<\infty$ and $\alpha\in\R$, write
\begin{equation*}
	\Phi_{p,\alpha}(t)
	\coloneqq t^p\log^\alpha(e+t).
\end{equation*}
We write $L^p(\log L)^\alpha(Q)$ for the corresponding local Orlicz space.
When $\alpha<0$, $\Phi_{p,\alpha}$ may be modified near the origin to give
an equivalent Young function; this does not alter the associated local norm
up to equivalence.  If $\alpha\geq0$, then
\begin{equation*}
	\overline{\Phi_{p,\alpha}}(t)
	\approx 
	t^{p'}\log^{-\alpha/(p-1)}(e+t)
	\qquad (t\ \text{sufficiently large}).
\end{equation*}
The next lemma, due to Hu and Tao \cite[Lemma 2.1]{HuTao2023}, converts a
small $L^1$ bound and a uniform Orlicz bound into an intermediate $L^q$
estimate.

\begin{lemma}[\cite{HuTao2023}]\label{lem:orlicz-local-interpolation}
	Let $Q\subset\Rd$ be a cube, let $1<p<\infty$ and $\alpha\geq0$, and
	let $0<C_1\leq1$.  Suppose that $f$ is measurable on $Q$ and
	\begin{equation*}
		\frac{1}{\abs{Q}}
		\int_Q\abs{f(y)}\,\dd y
		\leq C_1,
		\quad
		\norm{f}_{L^p(\log L)^{-\alpha},Q}\leq1.
	\end{equation*}
	Fix $1<q<p$, and let $r\in(0,1)$ be determined by
	\begin{equation*}
		\frac{1}{q}=r+\frac{1-r}{p}.
	\end{equation*}
	Then, for every $0<\eps<r$,
	\begin{equation*}
		\paren{
			\frac{1}{\abs{Q}}
			\int_Q\abs{f(y)}^q\,\dd y
		}^{1/q}
		\lesssim_{p,q,\alpha,\eps}
		C_1^{\eps}.
	\end{equation*}
\end{lemma}

We next introduce BMO space.
\begin{definition}\label{def:BMO}
	Let $b\in L_{\mathrm{loc}}^1(\Rd)$.  For a cube $Q\subset\Rd$, set
	\begin{equation*}
		b_Q\coloneqq\frac{1}{\abs{Q}}\int_Q b(x)\,\dd x.
	\end{equation*}
	The BMO seminorm of $b$ is
	\begin{equation*}
		\norm{b}_{\BMO}
		\coloneqq \sup_Q
		\frac{1}{\abs{Q}}
		\int_Q\abs{b(x)-b_Q}\,\dd x,
	\end{equation*}
	where the supremum is taken over all cubes $Q\subset\Rd$.  We write
	$b\in\BMO(\Rd)$ when $\norm{b}_{\BMO}<\infty$.
\end{definition}

We next recall the John-Nirenberg theorem, one of the most commonly used tools in the study of BMO spaces.

\begin{theorem}[John-Nirenberg theorem \cite{JohnNirenberg1961}]\label{thm:john-nirenberg}
	There exist constants $c_d,C_d>0$ such that, for every
	$b\in\BMO(\Rd)$, every cube $Q\subset\Rd$, and every $t>0$,
	\begin{equation*}
		\abs{\set{x\in Q:\abs{b(x)-b_Q}>t}}
		\leq
		C_d \exp\paren{-\frac{c_dt}{\norm{b}_{\BMO}}}
		\abs{Q}.
	\end{equation*}
\end{theorem}

An equivalent formulation is the local exponential integrability of BMO
oscillations.  We record a quantitative version convenient for later use.
\begin{corollary}\label{cor:bmo-intexp}
	There exist constants $c_d,C_d>0$ such that, for every
	$b\in\BMO(\Rd)$, every cube $Q\subset\Rd$, and every $s\geq0$ satisfying
	$s\norm{b}_{\BMO}\leq c_d$,
	\begin{equation*}
		\frac{1}{\abs{Q}}
		\int_Q\exp\paren{s\abs{b(x)-b_Q}}\,\dd x
		\leq 1+C_ds\norm{b}_{\BMO}.
	\end{equation*}
\end{corollary}

In terms of the localized Luxemburg norm, this gives the following familiar
estimate.
\begin{corollary}\label{cor:bmo-expL}
	For every $b\in\BMO(\Rd)$ and every cube $Q\subset\Rd$,
	\begin{equation*}
		\norm{b-b_Q}_{\exp L,Q}
		\lesssim
		\norm{b}_{\BMO}.
	\end{equation*}
\end{corollary}

The same distributional estimate yields uniform control of all finite
moments.
\begin{corollary}\label{cor:bmo-Lp}
	For every $1\leq p<\infty$, every $b\in\BMO(\Rd)$, and every cube
	$Q\subset\Rd$,
	\begin{equation*}
		\paren{
			\frac{1}{\abs{Q}}
			\int_Q\abs{b(x)-b_Q}^p\,\dd x
		}^{1/p}
		\lesssim_{d,p} \norm{b}_{\BMO}.
	\end{equation*}
\end{corollary}

We shall also need the following localized consequence.
\begin{corollary}\label{cor:bmo-subset-integral}
	Let $R\subset\Rd$ be a cube, and let $E\subset R$ be measurable with
	$0<\abs{E}\leq\abs{R}$.  Then, for every $b\in\BMO(\Rd)$,
	\begin{equation*}
		\int_E\abs{b(x)-b_R}\,\dd x
		\lesssim_d
		\norm{b}_{\BMO}\abs{E}
		\paren{1+\log\frac{\abs{R}}{\abs{E}}}.
	\end{equation*}
\end{corollary}

\begin{proof}
	Set $N=\norm{b}_{\BMO}$.  The assertion is immediate if $N=0$.
	Otherwise, the layer-cake representation and Theorem
	\ref{thm:john-nirenberg} give, for every $t_0>0$,
	\begin{align*}
		\int_E\abs{b(x)-b_R}\,\dd x
		&=\int_0^\infty
		\abs{E\cap\set{x\in R:\abs{b(x)-b_R}>t}}\,\dd t\\
		&\leq \abs{E}t_0
		+C_d\abs{R}\int_{t_0}^\infty
		\exp\paren{-\frac{c_dt}{N}}\,\dd t\\
		&=\abs{E}t_0
		+\frac{C_dN}{c_d}\abs{R}
		\exp\paren{-\frac{c_dt_0}{N}}.
	\end{align*}
	Taking
	\begin{equation*}
		t_0=\frac{N}{c_d}
		\log\paren{\frac{C_d\abs{R}}{\abs{E}}}
	\end{equation*}
	proves the claim.
\end{proof}

Combining Corollary \ref{cor:bmo-expL} with Proposition
\ref{prop:orlicz-holder}, we obtain the BMO-Orlicz product estimate used
throughout the proof.

\begin{lemma}\label{lem:bmo-orlicz}
	For every $b\in\BMO(\Rd)$, every cube $Q\subset\Rd$, and every locally
	integrable function $h$,
	\begin{equation*}
		\frac{1}{\abs{Q}}
		\int_Q\abs{b(x)-b_Q}\abs{h(x)}\,\dd x
		\lesssim
		\norm{b}_{\BMO}\norm{h}_{L\log L,Q}.
	\end{equation*}
\end{lemma}

The analytic-family argument of Coifman, Rochberg, and Weiss
\cite{CoifmanRochbergWeiss1976} requires uniform control of the exponential
weights generated by a BMO function.  The precise statement needed below is
as follows.
\begin{lemma}\label{lem:bmo-exponential-weight}
	Let $1<p<\infty$.  There exist constants $c_{d,p},C_{d,p}>0$ such that,
	whenever $b\in\BMO(\Rd)$ and $z\in\mathbb C$ satisfy
	\begin{equation*}
		\abs{z}\norm{b}_{\BMO}\leq c_{d,p},
	\end{equation*}
	the weight
	\begin{equation*}
		w_z(x)\coloneqq e^{p\operatorname{Re}(z)b(x)}
	\end{equation*}
	belongs to $A_p$ and satisfies $[w_z]_{A_p}\leq C_{d,p}$.
\end{lemma}

\begin{proof}
	For each cube $Q$, the constant factor
	$e^{p\operatorname{Re}(z)b_Q}$ cancels from the $A_p$ expression.  Hence
	\begin{align*}
		&\paren{\frac{1}{\abs Q}
		\int_Qe^{p\operatorname{Re}(z)b(x)}\,\dd x}
		\paren{\frac{1}{\abs Q}
		\int_Qe^{-p'\operatorname{Re}(z)b(x)}\,\dd x}^{p-1}\\
		&\quad=
		\paren{\frac{1}{\abs Q}
		\int_Qe^{p\operatorname{Re}(z)(b-b_Q)}\,\dd x}
		\paren{\frac{1}{\abs Q}
		\int_Qe^{-p'\operatorname{Re}(z)(b-b_Q)}\,\dd x}^{p-1}.
	\end{align*}
	If $c_{d,p}$ is sufficiently small, Corollary
	\ref{cor:bmo-intexp}, applied with $s=p\abs z$ and $s=p'\abs z$,
	bounds both factors uniformly in $Q$.  Taking the supremum over $Q$
	proves the assertion.
\end{proof}

\subsection{Interpolation results}\label{subsec:interpolation-results}

The decomposition used later requires interpolation between an $L^1$ bound
with favorable scale decay and a uniform $L^3$ bound.  We first record the
elementary special case of the Riesz-Thorin theorem that will be used.

\begin{proposition}\label{prop:L1-L3-interpolation}
	Let $T$ be a linear operator on $\Rd$ such that
	\begin{equation*}
		\norm{Tf}_{L^1}\leq A_1\norm{f}_{L^1},
		\quad
		\norm{Tf}_{L^3}\leq A_3\norm{f}_{L^3}.
	\end{equation*}
	Then
	\begin{equation*}
		\norm{Tf}_{L^2}
		\leq
		A_1^{1/4}A_3^{3/4}\norm{f}_{L^2}.
	\end{equation*}
\end{proposition}

\begin{proof}
	Apply the Riesz-Thorin interpolation theorem with interpolation parameter
	$\theta=3/4$.
\end{proof}

We shall also use the Stein-Weiss interpolation theorem with change of
measures.

\begin{theorem}[Stein-Weiss interpolation theorem \cite{SteinWeiss1958}]\label{thm:stein-weiss}
	Let $T$ be a linear operator, and let $v_i,w_i$ be positive measurable
	weights for $i=0,1$.  Suppose that
	\begin{equation*}
		\norm{Tf}_{L^{p_i}(w_i)}
		\leq
		M_i\norm{f}_{L^{p_i}(v_i)},
		\quad
		1\leq p_i<\infty.
	\end{equation*}
	For $0<\theta<1$, define $p_\theta$ by
	\begin{equation*}
		\frac{1}{p_\theta}
		=\frac{1-\theta}{p_0}+\frac{\theta}{p_1},
	\end{equation*}
	and define
	\begin{align*}
		w_\theta
		& \coloneqq w_0^{p_\theta(1-\theta)/p_0}
		w_1^{p_\theta\theta/p_1},
		\label{eq:stein-weiss-output-weight}\\
		v_\theta
		& \coloneqq v_0^{p_\theta(1-\theta)/p_0}
		v_1^{p_\theta\theta/p_1}.
	\end{align*}
	Then
	\begin{equation*}
		\norm{Tf}_{L^{p_\theta}(w_\theta)}
		\leq
		M_0^{1-\theta}M_1^\theta
		\norm{f}_{L^{p_\theta}(v_\theta)}.
	\end{equation*}
\end{theorem}

\begin{remark}\label{rem:stein-weiss-p2}
	In the case $p_0=p_1=2$, one has $p_\theta=2$ and
	\begin{equation*}
		w_\theta=w_0^{1-\theta}w_1^\theta,
		\quad
		v_\theta=v_0^{1-\theta}v_1^\theta.
	\end{equation*}
	This is the form used in the high-frequency multiplier estimates.
\end{remark}

\section{Reduction of the main theorem and linearization}
\label{sec:reduction-linearization}

We begin by replacing the continuous truncations with dyadic ones and applying
the Calder\'on-Zygmund decomposition.  We then localize the operators in
physical space, split the rough kernel according to its size, and stratify the
relevant dyadic cubes.  A Rademacher-Menshov argument ultimately reduces the
maximal estimate to a signed $L^2$ inequality.
\vspace{0.3cm}

\noindent \textbf{Reduction from continuous to dyadic truncations.} Choose a radial function $\varphi\in C_c^{\infty}(\Rd)$ such that
\begin{equation*}\label{eq:phi-support-partition}
	\supp\varphi
	\subset\set{x\in\Rd:2^{-4}<\abs{x}<2^{-2}},
	\quad
	\sum_{j\in\Z}\varphi_j(x)=1,
\end{equation*}
where $\varphi_j(x) \coloneqq \varphi(2^{-j}x)$. Define
\begin{equation*}\label{eq:dyadic-kernels-operators}
	\begin{aligned}
		& \mathcal{K}_j(x) \coloneqq \varphi_j(x)\frac{\Omega(x)}{\abs{x}^{d}}, \\
		& T_jf \coloneqq \mathcal{K}_j\ast f,\\
		& T_{b,j}f \coloneqq [b,T_j]f.
	\end{aligned}
\end{equation*}
The associated dyadic maximal truncations are given by
\begin{equation*}
	\begin{aligned}
		\Tmc_{\ast}f(x) &  \coloneqq \sup_{l\in\Z} \bigg| \sum_{j\geq l}T_j f(x) \bigg|, \\
		\mathcal T_{b,\ast}f(x) &  \coloneqq \sup_{l\in\Z}
		\bigg|\sum_{j\geq l}T_{b,j}f(x) \bigg|.
	\end{aligned}
\end{equation*}
We also introduce the rough Hardy-Littlewood maximal commutator
\begin{equation*}
	M_{\Omega,b}f(x)
	\coloneqq \sup_{r>0}\frac{1}{r^d}
	\int_{B(x,r)}
	\abs{\Omega(x-y)}
	\abs{b(x)-b(y)}
	\abs{f(y)}\,\dd y.
\end{equation*}
The support of the partition of unity, together with a direct estimate on the
annulus containing the truncation radius, gives
\begin{equation}\label{eq:continuous-dyadic-pointwise}
	T_{\Omega,b,\ast}f(x)
	\lesssim
	M_{\Omega,b}f(x)+\mathcal T_{b,\ast}f(x).
\end{equation}
The weak-type estimate for $M_{\Omega,b}$ and
\eqref{eq:continuous-dyadic-pointwise} reduce the main theorem to the following
endpoint bound for $\mathcal T_{b,\ast}$.
\begin{theorem}\label{thm:dyadic}
	Let $d\geq2$.  Suppose that $\Omega$ is homogeneous of degree zero on
	$\Rd\setminus\{0\}$, has mean zero on $\Sd$, and belongs to
	$L(\log L)^2(\Sd)$.  If $b\in\BMO(\Rd)$, then, for every $\lambda>0$,
	\begin{equation*}
		\abs{\set{x\in\Rd:\mathcal T_{b,\ast}f(x)>\lambda}}
		\lesssim
		\int_{\Rd}
		\Phi\paren{
			\frac{\abs{f(x)}}{\lambda}
		}
		\,\dd x,
	\end{equation*}
	where $\Phi(t)=t\log(e+t)$.
\end{theorem}

\noindent \textbf{Dyadic localization in physical space.} Let the standard dyadic grid be
\begin{equation*}\label{eq:standard-dyadic-grid}
	\D
	\coloneqq \set{
		2^j \big(\vec{m}+[0,1)^d\big):
		\vec{m}\in\Z^d,\ j\in\Z
	}.
\end{equation*}
For $\vec w\in\{0,\frac12\}^d$, define the shifted dyadic grid
\begin{equation*}\label{eq:shifted-dyadic-grid}
	\D^{\vec w}
	\coloneqq \set{
		2^j\big(\vec{m}+\vec w+[0,1)^d \big):
		\vec{m}\in\Z^d,\ j\in\Z
	}.
\end{equation*}
For a cube $J$, let $\frac12J$ denote the concentric cube with side length
$\ell(J)/2$.  For each $j\in\Z$, one has
\begin{equation}\label{eq:shifted-grid-partition}
	\sum_{\vec w\in\{0,\frac12\}^d}
	\sum_{\substack{J\in\D^{\vec w}\\\ell(J)=2^j}}
	\chi_{\frac12 J}=1.
\end{equation}
For $J\in\D^{\vec w}$ with $\ell(J)=2^j$, define the localized operators
\begin{align*}
	T_Jf(x)
	& \coloneqq \int_{\Rd}
	\mathcal{K}_j(x-y)f(y)\chi_{\frac12 J}(y)\,\dd y, \\
	T_{b,J}f(x)
	& \coloneqq \int_{\Rd} \paren{b(x)-b(y)} 	\mathcal{K}_j(x-y)
	f(y)\chi_{\frac12 J}(y)\,\dd y.
\end{align*}
It follows from \eqref{eq:shifted-grid-partition} that
\begin{equation}\label{eq:localized-operator-sum}
	\begin{aligned}
		T_j f & =\sum_{\vec w\in\{0,\frac12\}^d}
		\sum_{\substack{J\in\D^{\vec w}\\\ell(J)=2^j}}
		T_J f ,\\
		T_{b,j}f & =\sum_{\vec w\in\{0,\frac12\}^d}
		\sum_{\substack{J\in\D^{\vec w}\\\ell(J)=2^j}}
		T_{b,J}f.
	\end{aligned}
\end{equation}
Moreover, the support of $\mathcal{K}_j$ and the distance from $\frac12 J$ to $J^c$ imply that
\begin{equation}\label{eq:localized-support}
	\begin{aligned}
		\supp T_Jf & \subset J, \\
		\supp T_{b,J}f & \subset J.
	\end{aligned}
\end{equation}
The support property \eqref{eq:localized-support} permits a decomposition into
nested chains of cubes, which will be used to linearize the maximal operator.

\noindent \textbf{Calder\'on-Zygmund decomposition and reduction of the bad part.} By homogeneity, it suffices to consider $\lambda=1$.  Apply the
Calder\'on-Zygmund decomposition to $f$, relative to the standard dyadic grid
$\D$, at height $1$.  Thus $f=g+h$, where
\begin{align*}
	g
	& \coloneqq f\chi_{(\bigcup_{Q\in\Qmc}Q)^c}
	+\sum_{Q\in\Qmc}f_Q\chi_Q,\\
	h
	& \coloneqq \sum_{Q\in\Qmc}h_Q,
	\\
	h_Q \coloneqq \paren{f-f_Q}\chi_Q,
\end{align*}
and $\mathcal Q$ is the collection of bad cubes.  The $L^2$ boundedness of the
maximal truncated commutator, Chebyshev's inequality, and the standard
properties of this decomposition give
\begin{equation*}
	\abs{\set{x:\mathcal T_{b,\ast}g(x)>1/2}}
	\lesssim \norm{\mathcal T_{b,\ast}g}_{L^2}^2
	\lesssim
	\Bnorm^2\|f\|_{L^1}.
\end{equation*}
It remains to control the contribution of the bad part $h$, namely,
\begin{equation*}
	\abs{\set{x:\mathcal T_{b,\ast}h(x)>1/2}}.
\end{equation*}
For each $Q\in\Qmc$, write
\begin{equation*}
	b(x)-b(y)
	=\paren{b(x)-b_Q}-\paren{b(y)-b_Q}.
\end{equation*}
This yields the pointwise decomposition
\begin{equation}
	\label{eq:bad-commutator-two-parts}
	\begin{aligned}
		\mathcal T_{b,\ast}h(x)
		&\leq
		\sup_{l\in\Z} \bigg| 
		\sum_{j \geq l}\sum_{Q\in\Qmc}
		\paren{b(x)-b_Q}T_jh_Q(x)
		\bigg| \\
		&\quad+
		\Tmc_{\ast} \Big( 
		\sum_{Q\in\Qmc}\paren{b-b_Q}h_Q
		\Big)(x).
	\end{aligned}
\end{equation}

The weak $(1,1)$ bound for the dyadic maximal truncation of the rough singular
integral \cite{Lai2025}, together with Lemma \ref{lem:bmo-orlicz}, gives
\begin{align*}
	&\bigg| \bigg\{ x:\Tmc_{\ast} \bigg( 
	\sum_{Q\in\Qmc}(b-b_Q)h_Q \bigg)
	(x)>1/4 \bigg\} \bigg|
	\notag\\
	&\quad\lesssim
	\sum_{Q\in\Qmc}
	\norm{(b-b_Q)h_Q}_1
	\notag\\
	&\quad\lesssim
	\Bnorm
	\sum_{Q\in\Qmc}\abs Q\,
	\norm{h_Q}_{L\log L,Q}
	\notag\\
	&\quad\lesssim
	\Bnorm\|\Phi(|f|)\|_{L^1}.
\end{align*}
The last inequality uses the Calder\'on-Zygmund decomposition and the bound
$1<|f|_Q\leq\Phi(|f|)_Q$.  It remains to estimate the first term in
\eqref{eq:bad-commutator-two-parts}, which contains
\begin{equation*}
	\paren{b(x)-b_Q}T_jh_Q(x).
\end{equation*}
More precisely, we must estimate
\begin{equation*}
	\bigg| \bigg\{x: \sup_{l\in\Z} \Big| 
	\sum_{j \geq l}\sum_{Q\in\Qmc}
	\paren{b(x)-b_Q}T_jh_Q(x)
	\Big| > 1/4 \bigg\} \bigg|.
\end{equation*}
Set
\begin{equation*}
	\begin{aligned}
		\Qmc_j &  \coloneqq \set{Q\in\Qmc:\ell(Q)=2^j}, \\
		h_j &  \coloneqq \sum_{Q\in\Qmc_j}h_Q.
	\end{aligned}
\end{equation*}
Define
\begin{equation*}
	E^{\ast} \coloneqq \bigcup_{Q\in\Qmc}2^{300}Q.
\end{equation*}
By the properties of the Calder\'on-Zygmund decomposition,
\begin{equation*}
	\abs{E^{\ast}}\lesssim\|f\|_{L^1}.
\end{equation*}
If $x\notin E^{\ast}$, the support of the kernel implies that
\begin{equation*}
	T_jh_{j-s}(x)=0,
\end{equation*}
whenever $s<200$.  Consequently, \eqref{eq:localized-operator-sum} gives, for
every $x\notin E^{\ast}$,
\begin{align*}
	&\sup_{l\in\Z} \bigg| 
	\sum_{j \geq l}\sum_{Q\in\Qmc}
	\paren{b(x)-b_Q}T_jh_Q(x) \bigg| 
	\notag\\
	&\quad\leq
	\sum_{\vec w\in\{0,\frac12\}^d}
	\sum_{s\geq200}
	\sup_{l\in\Z} \bigg| \sum_{\substack{J\in\D^{\vec w}\\\ell(J)\geq2^l}}
	\sum_{\substack{Q\in\Qmc_{j-s}\\Q\subset\frac12 J}}
	\paren{b(x)-b_Q}T_Jh_Q(x) \bigg|,
\end{align*}
Here $j=j(J)$ is determined by $\ell(J)=2^j$.  Since there are only $2^d$
possible shifts, we fix $\vec w\in\{0,\frac12\}^d$ and reduce the problem to
estimating
\begin{equation}\label{eq:fixed-shifted-grid-goal}
	\bigg| \bigg\{x:
	\sum_{s\geq200}
	\sup_{l\in\Z} \bigg| \sum_{\substack{J\in\D^{\vec w}\\\ell(J)\geq2^l}}
	\sum_{\substack{Q\in\Qmc_{j-s}\\Q\subset\frac12 J}}
	\paren{b(x)-b_Q}T_J h_Q(x) \bigg| > 2^{-d-2}\bigg\} \bigg|.
\end{equation}

\noindent \textbf{Decomposition of the rough kernel by size.} Fix a sufficiently small constant $\eta>0$. For each $s\geq200$, decompose $\Omega$ as
\begin{align*}
	\Omega_\infty(\theta)
	& \coloneqq \Omega(\theta)
	\chi_{\{\abs{\Omega(\theta)}>2^{\eta s}\}},
	\\
	\Omega_0(\theta)
	& \coloneqq \Omega(\theta)
	\chi_{\{\abs{\Omega(\theta)}\leq2^{\eta s}\}}.
	\label{eq:omega-small-part}
\end{align*}
Correspondingly, write
\begin{equation*}
	\mathcal{K}_j=\mathcal{K}_{j,\infty}+\mathcal{K}_{j,0},
\end{equation*}
and let $T_{J,i}$ denote the operator obtained from $T_J$ by replacing
$\mathcal K_j$ with $\mathcal K_{j,i}$.  Thus
\begin{equation*}
	T_J=T_{J,\infty}+T_{J,0}.
\end{equation*}

The large-amplitude part can be summed directly in $L^1$.

\begin{proposition}
	\label{prop:large-omega-L1}
	For each fixed $\vec w\in\{0,\frac12\}^d$,
	\begin{align*}
		&\sum_{s\geq200}
		\sum_{J\in\D^{\vec w}}
		\sum_{\substack{Q\in\Qmc_{j-s}\\Q\subset\frac12 J}}
		\norm{
			\paren{b-b_Q}T_{J,\infty}h_Q
		}_1 \lesssim
		\Bnorm\norm f_1.
	\end{align*}
\end{proposition}

\begin{proof}
	For $m\geq1$, let
	\begin{equation*}
		\Omega_m(\theta)
		\coloneqq \Omega(\theta)
		\chi_{\{2^{m-1}<\abs{\Omega(\theta)}\leq2^m\}},
	\end{equation*}
	and set
	\begin{equation*}
		\mathcal{K}_{j,m}(x)
		\coloneqq  \varphi_j(x)\frac{\Omega_m(x)}{\abs{x}^d}.
	\end{equation*}
	For $Q\in\Qmc_{j-s}$ with $Q\subset\frac12 J$, we have
	\begin{align*}
		&\norm{(b-b_Q)\mathcal{K}_{j,m}\ast h_Q}_1
		\notag\\
		&\quad\leq
		\int_J\int_Q
		\abs{\mathcal{K}_{j,m}(x-y)}
		\abs{h_Q(y)}
		\paren{
			\abs{b(x)-b_J}+\abs{b_J-b_Q}
		}
		\,\dd y\,\dd x
		\notag\\
		&\quad=:I_1(J,Q,m)+I_2(J,Q,m).
	\end{align*}
	
	We first estimate $I_1$. The Orlicz-H\"older inequality and the John-Nirenberg theorem imply that
	\begin{align*}
		I_1(J,Q,m) \lesssim
		\Bnorm\norm{h_Q}_1
		\sup_{y\in Q}
		\norm{\Omega_m(\,\cdot-y)}_{L\log L,J}.
	\end{align*}
	Using polar coordinates and the fact that $y \in Q$, we obtain
	\begin{equation*}
		\begin{aligned}
			& \qquad \left\|\Omega_m(\cdot-y)\right\|_{L\log L,J} \\
			& =
			\inf\left\{
			\lambda>0:
			\frac{1}{|J|}
			\int_J
			\frac{|\Omega_m(x-y)|}{\lambda}
			\log\left(
			e+\frac{|\Omega_m(x-y)|}{\lambda}
			\right)\dd x
			\le 1
			\right\}
			\\
			&\lesssim
			\inf\left\{
			\lambda>0:
			\frac{\|\Omega_m\|_{L^1}}{\lambda}
			\log\left(
			e+\frac{\|\Omega_m\|_{L^\infty}}{\lambda}
			\right)
			\le 1
			\right\}
			\\
			&\lesssim
			\|\Omega_m\|_{L^\infty}^{-1} + \|\Omega_m\|_{L^1}
			\log\left(
			e+\|\Omega_m\|_{L^\infty}
			\right)
			\\
			&\lesssim
			2^{-m}
			+
			m\|\Omega_m\|_{L^1},
			\ m\ge \eta s>0.
		\end{aligned}
	\end{equation*}	
	Hence
	\begin{equation}\label{eq:I1-final-bound}
		I_1(J,Q,m)
		\lesssim
		\paren{2^{-m}+m\norm{\Omega_m}_1}
		\Bnorm\norm{h_Q}_1.
	\end{equation}
	
	On the other hand, $Q\subset\frac12 J$ and $\ell(Q)=2^{-s}\ell(J)$ imply that
	\begin{equation*}
		\abs{b_J-b_Q}\lesssim s\Bnorm.
	\end{equation*}
	Consequently,
	\begin{equation}\label{eq:I2-final-bound}
		I_2(J,Q,m)
		\lesssim
		s\Bnorm\norm{\Omega_m}_1\norm{h_Q}_1.
	\end{equation}
	
	Summing \eqref{eq:I1-final-bound} and \eqref{eq:I2-final-bound} over the relevant indices, and using the fact that $m\geq\eta s$ implies $s\lesssim m$, we conclude that
	\begin{align*}
		&\sum_{s\geq200}
		\sum_{J\in\D^{\vec w}}
		\sum_{\substack{Q\in\Qmc_{j-s}\\Q\subset\frac12 J}}
		\norm{(b-b_Q)T_{J,\infty}h_Q}_{L^1}
		\notag\\
		&\quad\lesssim \Bnorm \|f\|_{L^1}
		\brak{
			1+\sum_{m\geq1}
			m^2\norm{\Omega_m}_{L^1(\Sd)}
		}
		\notag\\
		& \quad \lesssim  C_{\Omega}\Bnorm \|f\|_{L^1}.
	\end{align*}
\end{proof}

Proposition \ref{prop:large-omega-L1} and Chebyshev's inequality give the
required endpoint bound for the large-amplitude part.  We therefore turn to
$\Omega_0$.  The following $L^2$ estimate is the central step in the reduction.

\begin{proposition}
	\label{prop:small-omega-key-L2}
	There exists a constant $\delta>0$ such that, for every $s\geq200$ and each fixed $\vec w\in\{0,\frac12\}^d$,
	\begin{align*}
		&\Bigg\|
		\sup_{l\in\Z} \bigg|
		\sum_{\substack{J\in\D^{\vec w}\\\ell(J)\geq2^l}}
		\sum_{\substack{Q\in\Qmc_{j-s}\\Q\subset\frac12 J}}
		\paren{b-b_Q}T_{J,0}h_Q
		\bigg|
		\Bigg\|_2 \lesssim 
		2^{-\delta s}\Bnorm\norm f_1^{1/2}.
	\end{align*}
\end{proposition}

Once Proposition \ref{prop:small-omega-key-L2} is established, summation in
$s$, followed by Chebyshev's inequality, controls the small-amplitude
contribution to \eqref{eq:fixed-shifted-grid-goal}.  It therefore remains to
prove this proposition.  Fix $s\geq200$ for the rest of the section.  To
simplify notation, write $\Omega$, $\mathcal K_j$, and $T_J$ for $\Omega_0$,
$\mathcal K_{j,0}$, and $T_{J,0}$, respectively.  Then
\begin{equation*}
	\norm{\Omega}_{\infty} \leq 2^{\eta s}.
\end{equation*}

\noindent \textbf{Some lemmas and stratification of dyadic cubes.} We first recall two standard tools for the frequency decomposition and the
linearization of the maximal operator.  The subsequent stratification of
dyadic cubes follows Lai \cite{Lai2025}.

\begin{lemma}[Weighted Mikhlin multiplier theorem]
	\label{lem:hormander-mihlin-weighted-sec3}
	Let \(m\) be a bounded complex-valued function on
	\(\mathbb{R}^d\setminus\{0\}\). Suppose that, for every multi-index $\alpha$ with $\abs\alpha \leq d$,
	\[
	|\partial_\xi^\alpha m(\xi)|
	\le A |\xi|^{-|\alpha|}.
	\]
	Let \(T_m\) be the Fourier multiplier operator defined by
	\[
	\widehat{T_m f}(\xi)
	=m(\xi)\widehat f(\xi).
	\]
	Then, for every \(w\in A_p(\mathbb{R}^d)\), $1 < p <\infty$,
	\[
	\|T_m f\|_{L^p(w)}
	\lesssim_{[w]_{A_p}} \bigl(\|m\|_{L^\infty}+A\bigr)
	\|f\|_{L^p(w)}.
	\]
\end{lemma}
For $w=1$, Lemma \ref{lem:hormander-mihlin-weighted-sec3} gives the usual
unweighted multiplier theorem.

\begin{lemma}[Rademacher-Menshov theorem]
	\label{lem:rademacher-menshov-sec3}
	Let $(X,\mu)$ be a measure space and let $f_1,\ldots,f_N$ be measurable functions on $X$. Suppose that, for every choice of $\eps_j\in\{-1,1\}$,
	\begin{equation*}
		\bigg\| \sum_{j=1}^N\eps_jf_j \bigg\|_{L^2(X)}\leq B.
	\end{equation*}
	Then
	\begin{equation*}
		\bigg\|
		\sup_{1\leq M\leq N}
		\Big|\sum_{j=1}^Mf_j \Big| \bigg\|_{L^2(X)}
		\lesssim B\log(2+N).
	\end{equation*}
\end{lemma}

For fixed $s$ and $\vec w$, let the collection of relevant localization cubes be
\begin{equation*}
	\mathcal H_s
	\coloneqq \set{
		J\in\D^{\vec w}:
		\text{for some }Q\in\Qmc_{j-s}, Q\subset\tfrac12 J
	}.
\end{equation*}
Lai's estimate \cite[Lemma 3.3]{Lai2025} gives the following Carleson-type
bounds for this collection.
\begin{lemma}[\cite{Lai2025}]
	\label{lem:relevant-cubes-carleson}
	For every measurable set $A\subset\Rd$,
	\begin{equation*}
		\sum_{\substack{J\in\mathcal H_s\\J\subset A}}
		\abs J \lesssim 2^{ds} \abs A.
	\end{equation*}
	Moreover,
	\begin{equation*}
		\sum_{J\in\mathcal H_s}\abs J \lesssim 2^{ds}\|f\|_{L^1}.
	\end{equation*}
\end{lemma}

The shifted grid $\D^{\vec w}$ need not be nested across scales.  We therefore
bisect each $J\in\D^{\vec w}$ in every coordinate direction and enumerate the
resulting $2^d$ subcubes, in a fixed spatial order, as
\begin{equation*}
	J^1,J^2,\ldots,J^{2^d}.
\end{equation*}
For each fixed $1\leq\iota\leq2^d$, any two cubes of the form $J^\iota$ are either disjoint or nested. Moreover, if $I^\iota\subset J^\iota$, then $I\subset J$.

Fix $\iota$. Set $F_{s,\iota}^0 \coloneqq \Rd$ and define recursively
\begin{equation*}
	F_{s,\iota}^n
	\coloneqq \Bigg\{
	x\in\Rd:
	\sum_{\substack{J\in\mathcal H_s\\J^\iota\subset F_{s,\iota}^{n-1}}}
	\chi_{J^\iota}(x)>C_0 2^{ds}
	\Bigg\},
	\quad n\geq1,
\end{equation*}
where $C_0$ is a sufficiently large constant depending only on $d$. By construction,
\begin{equation*}
	F_{s,\iota}^1\supset F_{s,\iota}^2
	\supset\cdots\supset F_{s,\iota}^n\supset\cdots.
\end{equation*}

The sets $F_{s,\iota}^n$ satisfy the following decay estimate.
\begin{lemma}[\cite{Lai2025}]
	\label{lem:F-measure-decay}
	For every $n\geq1$,
	\begin{equation*}
		\abs{F_{s,\iota}^n}
		\lesssim 2^{-2n}\|f\|_{L^1}.
	\end{equation*}
\end{lemma}

We now use $F_{s,\iota}^n$ to stratify $\mathcal H_s$. Define
\begin{align*}
	\mathcal I_{s,\iota}^{\#,1}
	& \coloneqq \set{
		J\in\mathcal H_s:J^\iota\not\subset F_{s,\iota}^1
	},\\
	\mathcal I_{s,\iota}^{\#,n}
	& \coloneqq \set{
		J\in\mathcal H_s:
		J^\iota\subset F_{s,\iota}^{n-1},\ 
		J^\iota\not\subset F_{s,\iota}^{n}
	},
	\quad n\geq2.
\end{align*}
For each fixed $\iota$, these collections form a partition of $\mathcal H_s$.

Summing also over the $2^d$ children of each cube gives the identity
\begin{align*}
	\qquad \sum_{J\in\mathcal H_s}
	\sum_{\substack{Q\in\Qmc_{j-s}\\Q\subset\frac12 J}}
	\paren{b-b_Q}T_Jh_Q =
	\sum_{\iota=1}^{2^d}
	\sum_{n\geq1}
	\sum_{J\in\mathcal I_{s,\iota}^{\#,n}}
	\sum_{\substack{Q\in\Qmc_{j-s}\\Q\subset\frac12 J}}
	\paren{b-b_Q}T_Jh_Q\chi_{J^\iota}.
	\label{eq:cube-stratification-identity}
\end{align*}
Consequently, Proposition \ref{prop:small-omega-key-L2} follows from the following estimate on each stratum.
\begin{proposition}
	\label{prop:fixed-layer-maximal}
	There exists $\delta>0$ such that, for every $n\geq1$ and $1\leq\iota\leq2^d$,
	\begin{align*}
		&\Bigg\|
		\sup_{l\in\Z}
		\bigg| 
		\sum_{\substack{J\in\mathcal I_{s,\iota}^{\#,n}\\\ell(J)\geq2^l}}
		\sum_{\substack{Q\in\Qmc_{j-s}\\Q\subset\frac12 J}}
		\paren{b-b_Q}T_Jh_Q\chi_{J^\iota}
		\bigg|
		\Bigg\|_2
		\notag\\
		&\quad\lesssim 2^{-n} 2^{-\delta s}
		\Bnorm\norm f_1^{1/2}.
	\end{align*}
\end{proposition}

\noindent \textbf{Linearization of the maximal operator.} We now linearize the maximal operator, following Lai \cite{Lai2025}.  Fix
$s,n$, and $\iota$, and set
\begin{equation*}
	\mathcal L_{s,\iota}^{\#,n}
	\coloneqq \set{J^\iota:J\in\mathcal I_{s,\iota}^{\#,n}}.
\end{equation*}
Any two cubes in this collection are either disjoint or nested. For a collection of cubes $\mathscr A$, let $\operatorname{Max}(\mathscr A)$ denote the subcollection of cubes maximal in $\mathscr A$ with respect to inclusion. Define
\begin{equation*}
	\mathcal L_1
	\coloneqq \operatorname{Max} \big( \mathcal L_{s,\iota}^{\#,n} \big).
\end{equation*}
Recursively, define
\begin{equation*}
	\mathcal L_{u+1}
	\coloneqq \operatorname{Max} \Big( 
	\mathcal L_{s,\iota}^{\#,n}
	\setminus\bigcup_{v=1}^{u}\mathcal L_v \Big).
\end{equation*}

Set
\begin{equation*}
	u_0 \coloneqq C_02^{ds}.
\end{equation*}
By the definition of $F_{s,\iota}^n$, each
$J^\iota\in\mathcal L_{s,\iota}^{\#,n}$ has at most $u_0$ strict ancestors in
$\mathcal L_{s,\iota}^{\#,n}$ that lie in $F_{s,\iota}^{n-1}$. Hence
\begin{equation*}
	\mathcal L_u=\varnothing,
	\quad u\geq u_0+1.
\end{equation*}
For $1\leq u\leq u_0$, define
\begin{equation*}
	\mathcal M_u
	\coloneqq \set{J:J^\iota\in\mathcal L_u}.
\end{equation*}
Cubes in the same generation $\mathcal L_u$ are pairwise disjoint. For each fixed $x$, the relevant cubes containing $x$ form a nested chain, from larger to smaller cubes as $u$ increases. Ordering the same chain by increasing side length $\ell(J)$ reverses this order. It follows that
\begin{align*}
	&\sup_{l\in\Z}
	\Bigg| 
	\sum_{\substack{J\in\mathcal I_{s,\iota}^{\#,n}\\\ell(J)\geq2^l}}
	\sum_{\substack{Q\in\Qmc_{j-s}\\Q\subset\frac12 J}}
	\paren{b-b_Q}T_Jh_Q\chi_{J^\iota}
	\Bigg| \\
	&\quad=
	\sup_{1\leq v\leq u_0}
	\Bigg| \sum_{u=1}^{v} \sum_{J\in\mathcal M_u}
	\sum_{\substack{Q\in\Qmc_{j-s}\\Q\subset\frac12 J}}
	\paren{b(x)-b_Q}T_Jh_Q(x)\chi_{J^\iota}(x) \Bigg|.
\end{align*}
Thus the maximal truncation becomes the maximal partial sum of a finite
sequence.  Lemma \ref{lem:rademacher-menshov-sec3} reduces the problem to the
following signed estimate.

\begin{proposition}
	\label{prop:sec3-linearized-signed-L2}
	There exists $\delta>0$ such that, for every choice of signs
	$\eps_1,\ldots,\eps_{u_0}\in\{-1,1\}$,
	\begin{align*}
		\Bigg\| \sum_{u=1}^{u_0}\eps_u \sum_{J\in\mathcal M_u}
		\sum_{\substack{Q\in\Qmc_{j-s}\\Q\subset\frac12 J}}
		\paren{b-b_Q}T_Jh_Q\chi_{J^\iota}\Bigg\|_{L^2} \lesssim 2^{-n}2^{-\delta s/2}
		\Bnorm \|f\|_{L^1}^{1/2}.
	\end{align*}
\end{proposition}

Indeed, Proposition \ref{prop:sec3-linearized-signed-L2} and the
Rademacher-Menshov theorem introduce only a factor
$\log(2+u_0)\lesssim s$.  After decreasing $\delta$, this factor is absorbed by
the exponential decay in $s$, yielding Proposition
\ref{prop:fixed-layer-maximal}.  Summation in $n$ and $\iota$ then gives
Proposition \ref{prop:small-omega-key-L2}.

\section{Estimates for the linearized operator}
\label{sec:linearized-operator-estimate}

We now prove the signed estimate obtained in the preceding section.  Fix
$s\geq200$, $n\geq1$, $1\leq\iota\leq2^d$, and
$\vec w\in\{0,\frac12\}^d$.  Also fix signs
$\varepsilon_1,\ldots,\varepsilon_{u_0}\in\{-1,1\}$, where
$u_0=C_02^{ds}$.  We first reindex the operator by the scale cubes $J$ and the
Calder\'on-Zygmund bad cubes $Q$.  We then smooth the spatial cutoff and
decompose the resulting operator microlocally.

Recall the expression on the left-hand side of Proposition \ref{prop:sec3-linearized-signed-L2}:
\begin{equation}\label{eq:sec4-beta-u}
	\sum_{u=1}^{u_0}\eps_u\sum_{J\in\mathcal M_u} \sum_{\substack{Q\in\mathcal Q_{j-s}\\Q\subset\frac12J}}
	\big(b(x)-b_Q\big)T_Jh_Q(x)\chi_{J^\iota}(x),
	\ \ell(J)=2^j.
\end{equation}
Throughout this section, $\ell(J)=2^j$ and $\ell(I)=2^i$ unless stated
otherwise.  Thus uppercase letters denote cubes, while the corresponding
lowercase letters denote their scales.  For $J\in\mathcal M_u$, set
\begin{equation*}
	\varepsilon_J \coloneqq \varepsilon_u.
\end{equation*}
Since $\mathcal M_1,\ldots,\mathcal M_{u_0}$ partition $\mathcal I_{s,\iota}^{\#,n}$, we may rewrite \eqref{eq:sec4-beta-u} as
\begin{align*}
	\sum_{\substack{J\in\mathcal I_{s,\iota}^{\#,n}}}
	\sum_{\substack{Q\in\mathcal Q_{j-s}\\Q\subset\frac12J}}
	\varepsilon_J\big(b(x)-b_Q\big)
	T_Jh_Q(x)\chi_{J^\iota}(x).
\end{align*}
Because the sum is restricted to cubes with $Q \subset \frac{1}{2} J$,
\begin{equation*}
	T_Jh_Q=T_jh_Q.
\end{equation*}
Hence \eqref{eq:sec4-beta-u} becomes
\begin{align*}
	\sum_{\substack{J\in\mathcal I_{s,\iota}^{\#,n}}}
	\sum_{\substack{Q\in\mathcal Q_{j-s}\\Q\subset\frac12J}}
	\varepsilon_J\big(b(x)-b_Q\big)
	T_j h_Q(x)\chi_{J^\iota}(x).
\end{align*}
To simplify the sums, define the collection of relevant bad cubes by
\begin{equation*}
	\mathfrak{Q}_{J-s}
	\coloneqq \Big\{
	Q\in\mathcal Q_{j-s}:
	Q\subset\frac12J \Big\}.
\end{equation*}
With this notation, \eqref{eq:sec4-beta-u} takes the form
\begin{equation*}
	\sum_{\substack{J\in\mathcal I_{s,\iota}^{\#,n}}}
	\sum_{\substack{Q\in\mathfrak{Q}_{J-s}}}
	\varepsilon_J\big(b(x)-b_Q\big)
	T_j h_Q(x)\chi_{J^\iota}(x).
\end{equation*}
After relabeling the decay exponent, Proposition
\ref{prop:sec3-linearized-signed-L2} is equivalent to the following statement.
\begin{proposition}
	\label{prop:sec4-reindexed-linearized-estimate}
	There exists a constant $\delta>0$ such that, for any choice of signs $\varepsilon_1,\ldots,\varepsilon_{u_0}\in\{-1,1\}$,
	\begin{equation*}
		\Big\|
		\sum_{J \in \mathcal{I}_{s,\iota}^{\#,n}} \sum_{Q\in\mathfrak{Q}_{J-s}}
		\varepsilon_J\big(b-b_Q\big)
		\big(T_jh_Q\big)\chi_{J^\iota}
		\Big\|_{L^2(\mathbb R^d)}
		\lesssim
		2^{-n}2^{-\delta s}
		\|b\|_{\operatorname{BMO}(\mathbb R^d)}
		\|f\|_{L^1(\mathbb R^d)}^{1/2},
	\end{equation*}
\end{proposition}

We shall repeatedly use the following bound for the bad functions in a fixed
stratum.
\begin{lemma}[\cite{Lai2025}]
	\label{lem:layered-bad-mass}
	For fixed $s\geq 200$, $1\leq\iota\leq 2^d$, and $n\geq 1$,
	\begin{equation*}
		\begin{aligned}
			& \sum_{\substack{
					J\in\mathcal I_{s,\iota}^{\#,n}		}}
			\sum_{Q\in\mathfrak Q_{J-s}}
			\|h_Q\|_{L^1(\mathbb R^d)} \lesssim
			2^{-2n}\|f\|_{L^1(\mathbb R^d)}.
		\end{aligned}
	\end{equation*}
\end{lemma}

\subsection{Smoothing and microlocal decomposition}
\label{subsec:smoothing-microlocal-decomposition}

We replace $\chi_{J^\iota}$ in Proposition
\ref{prop:sec4-reindexed-linearized-estimate} by a smooth cutoff, thereby
isolating an error supported near $\partial J^\iota$.  We then separate the
relevant frequency scales and decompose the high-frequency term according to
the directions of the kernel and the Fourier variable.

Choose a nonnegative radial function $\varpi\in C_c^\infty(\mathbb R^d)$ such
that
\begin{align*}
	& \operatorname{supp}\varpi \subset\big\{x\in\mathbb R^d:|x|\leq2^{-5}\big\}, \\
	& \qquad \int_{\mathbb R^d}\varpi(x)\,\dd x=1.
\end{align*}
For $t\in\mathbb R$, set
\begin{align*}
	\varpi_t(x) &  \coloneqq 2^{-td}\varpi(2^{-t}x), \\
	P_tg(x) &  \coloneqq \int_{\mathbb R^d}\varpi_t(x-z)g(z)\,\dd z.
\end{align*}
Fix a constant $\kappa\in(0,1)$, to be chosen later, and define
\begin{equation*}
	\widetilde\chi_{J^\iota}(x)
	\coloneqq P_{j-s\kappa} \chi_{J^\iota}(x).
\end{equation*}
By the nonnegativity and normalization of $\varpi$,
\begin{align*}
	& \qquad  0\leq\widetilde\chi_{J^\iota}\leq1, \\
	& \operatorname{supp}\widetilde\chi_{J^\iota}
	\subset
	\Big\{x:\operatorname{dist}(x,J^\iota)
	\leq2^{j-s\kappa-5}\Big\}.
\end{align*}
Moreover, for every multi-index $|\alpha| \leq N$,
\begin{equation*}
	|\partial^\alpha\widetilde\chi_{J^\iota}(x)|
	\lesssim
	2^{-(j-s\kappa)|\alpha|}.
\end{equation*}
Also,
\begin{equation*}
	\operatorname{supp}\big(\chi_{J^\iota}
	-\widetilde\chi_{J^\iota}\big)
	\subset
	\Big\{x:\operatorname{dist}(x,\partial J^\iota)
	\leq2^{j-s\kappa-5}\Big\}.
\end{equation*}
Thus $\chi_{J^\iota}-\widetilde\chi_{J^\iota}$ is supported in a boundary layer of thickness comparable to $2^{j-s\kappa}$.

We next decompose the kernel microlocally, following Seeger
\cite{Seeger1996}.  Choose $\gamma$ so that
\begin{equation*}
	0<\gamma<\kappa<1.
\end{equation*}
Choose a maximal $2^{-s\gamma-4}$-separated set of points in $\mathbb S^{d-1}$,
\begin{equation*}
	\Theta_s=\big\{e_{\nu,s}\}_\nu,
\end{equation*}
with the following properties:
\begin{enumerate}[(i)]
	\item 	$|e_{\nu,s}-e_{\nu',s}| \geq 2^{-s\gamma-4}$,
	\item For every $\theta \in \Sd$, there exists $e_{\nu,s}$ such that $|\theta-e_{\nu,s}| \leq2^{-s\gamma-4}$.
\end{enumerate}
A volume comparison gives
\begin{equation*}
	\card(\Theta_s)\lesssim2^{s\gamma(d-1)}.
\end{equation*}
We may further construct pairwise disjoint measurable sets $E_{\nu,s}\subset\mathbb S^{d-1}$ such that
\begin{enumerate}[(i)]
	\item $e_{\nu,s}\in E_{\nu,s}$,
	\item $\operatorname{diam}(E_{\nu,s})\leq2^{-s\gamma-2}$,
	\item $\mathbb S^{d-1}=\bigcup_{\nu}E_{\nu,s}$.
\end{enumerate}
For simplicity, we suppress the dependence on $s$ in $e_{\nu,s}$ and $E_{\nu,s}$, writing $e_\nu$ and $E_\nu$. For each fixed $\nu$, define the directionally localized kernel and operator by
\begin{align*}
	\mathcal{K}_{j,\nu}(x)
	&  \coloneqq \mathcal{K}_j(x)\chi_{E_\nu}\Big(\frac{x}{|x|}\Big), \\
	T_{j,\nu} f(x)
	&  \coloneqq \int_{\mathbb R^d}\mathcal{K}_{j,\nu}(x-y)f(y)\,\dd y.
\end{align*}

Finally, introduce a directional cutoff in frequency space.  Choose a
real-valued, nonnegative, even function $\Psi\in C_c^\infty(\mathbb R)$ with
$0\leq\Psi\leq1$ and
\begin{equation*}
	\Psi(x)=
	\begin{cases}
		1, & |x| \leq 2,\\
		0, & |x| \geq 4.
	\end{cases}
\end{equation*}
Define the Fourier multiplier operator $G_\nu $ by
\begin{equation*}
	\widehat{G_\nu  f}(\xi)
	\coloneqq \Psi\Big(
	2^{s\gamma}
	\Big\langle e_\nu,\frac{\xi}{|\xi|}\Big\rangle
	\Big)\widehat f(\xi).
\end{equation*}

The decomposition rests on the identity
\begin{equation*}
	(b-b_Q)u
	=G_\nu\big[(b-b_Q)u\big]
	+[b,G_\nu]u
	+(b-b_Q)(I-G_\nu)u.
\end{equation*}
Together with the smoothing operators $P_t$, this identity reduces
Proposition \ref{prop:sec4-reindexed-linearized-estimate} to the following five
estimates.

\begin{lemma}
	\label{lem:sec42-boundary-error}
	There exists a constant $\delta_1>0$ such that
	\begin{align*}
		&\Big\|
		\sum_{J\in\mathcal I_{s,\iota}^{\#,n}}
		\sum_{Q\in\mathfrak{Q}_{J-s}}
		\varepsilon_J\big(b-b_Q\big)
		\big(T_jh_Q\big)
		\big(\chi_{J^\iota}-\widetilde\chi_{J^\iota}\big)
		\Big\|_{L^2(\mathbb R^d)} \lesssim 2^{-n}2^{-\delta_1s}
		\|b\|_{\operatorname{BMO}(\mathbb R^d)}
		\|f\|_{L^1(\mathbb R^d)}^{1/2}.
	\end{align*}
\end{lemma}

\begin{lemma}
	\label{lem:sec42-low-frequency}
	There exists a constant $\delta_2>0$ such that
	\begin{align*}
		&\Big\|
		\sum_{J\in\mathcal I_{s,\iota}^{\#,n}}
		\sum_{Q\in\mathfrak{Q}_{J-s}}
		\varepsilon_J\big(b-b_Q\big)
		P_{j-s\kappa}\Big[
		\big(T_jh_Q\big)\widetilde\chi_{J^\iota}
		\Big]
		\Big\|_{L^2(\mathbb R^d)} \lesssim
		2^{-n}2^{-\delta_2s}
		\|b\|_{\operatorname{BMO}(\mathbb R^d)}
		\|f\|_{L^1(\mathbb R^d)}^{1/2}.
	\end{align*}
\end{lemma}

\begin{lemma}
	\label{lem:sec42-main-microlocal}
	There exists a constant $\delta_3>0$ such that
	\begin{align*}
		& \Big\| \sum_{J\in\mathcal I_{s,\iota}^{\#,n}}
		\sum_{Q\in\mathfrak{Q}_{J-s}}
		\sum_{\nu}
		G_\nu \Big[
		\big(b-b_Q\big)
		\big(I-P_{j-s\kappa}\big)\Big[ \varepsilon_J
		\big(T_{j,\nu} h_Q\big)\widetilde\chi_{J^\iota}
		\Big]
		\Big]
		\Big\|_{L^2(\mathbb R^d)} \\
		& \qquad \lesssim 2^{-n}2^{-\delta_3s}
		\|b\|_{\operatorname{BMO}(\mathbb R^d)}
		\|f\|_{L^1(\mathbb R^d)}^{1/2}.
	\end{align*}
\end{lemma}

\begin{lemma}
	\label{lem:sec42-commutator-microlocal}
	There exists a constant $\delta_4>0$ such that
	\begin{align*}
		& \Big\|
		\sum_{J\in\mathcal I_{s,\iota}^{\#,n}}
		\sum_{Q\in\mathfrak{Q}_{J-s}}
		\sum_{\nu} [b,G_\nu] \big(I-P_{j-s\kappa}\big)\Big[ \varepsilon_J
		\big(T_{j,\nu} h_Q\big)\widetilde\chi_{J^\iota}
		\Big]
		\Big\|_{L^2(\mathbb R^d)} \\
		& \qquad \lesssim 2^{-n}2^{-\delta_4s}
		\|b\|_{\operatorname{BMO}(\mathbb R^d)}
		\|f\|_{L^1(\mathbb R^d)}^{1/2}.
	\end{align*}
\end{lemma}

\begin{lemma}
	\label{lem:sec42-remainder-microlocal}
	There exists a constant $\delta_5>0$ such that
	\begin{align*}
		& \Big\|
		\sum_{J\in\mathcal I_{s,\iota}^{\#,n}}
		\sum_{Q\in\mathfrak{Q}_{J-s}}
		\sum_{\nu} (b-b_Q)
		\big(I-G_\nu \big)
		\big(I-P_{j-s\kappa}\big)\Big[ \varepsilon_J
		\big(T_{j,\nu} h_Q\big)\widetilde\chi_{J^\iota}
		\Big]
		\Big\|_{L^2(\mathbb R^d)} \\
		& \qquad \lesssim 2^{-n}2^{-\delta_5s}
		\|b\|_{\operatorname{BMO}(\mathbb R^d)}
		\|f\|_{L^1(\mathbb R^d)}^{1/2}.
	\end{align*}
\end{lemma}

The boundary and low-frequency estimates are proved below.  Section
\ref{sec:proof-lemma45} establishes Lemma \ref{lem:sec42-main-microlocal} by a
$TT^*$ argument and Lemma \ref{lem:sec42-commutator-microlocal} by the analytic-family method of Coifman, Rochberg, and Weiss \cite{CoifmanRochbergWeiss1976}. The remainder is
estimated in Section \ref{sec:proof-lemma47}.

\subsection{Boundary and low-frequency estimates}
\label{subsec:boundary-low-frequency}

We first prove Lemma \ref{lem:sec42-boundary-error}.  Set
\begin{align*}
	\mathcal E_{s,n,\iota}(x)  \coloneqq  \sum_{J \in \mathcal{I}_{s,\iota}^{\#,n}}
	\sum_{Q\in\mathfrak{Q}_{J-s}}
	\varepsilon_J\big(b(x)-b_Q\big)
	T_jh_Q(x)
	\big(\chi_{J^\iota}(x)-\widetilde\chi_{J^\iota}(x)\big).
\end{align*}
We combine a rough $L^3$ estimate with an $L^1$ estimate having
$2^{-s\kappa}$ decay.  In what follows, $\sum_{J,Q}$ denotes
$\sum_{J\in\mathcal I_{s,\iota}^{\#,n}}
\sum_{Q\in\mathfrak Q_{J-s}}$.

\noindent
\textbf{\(L^3\) estimate.}
Let
\begin{equation*}
	\begin{aligned}
		W(y) \coloneqq \sum_{J,Q} \frac{\|h_Q\|_{L^1}}{|Q|}\chi_Q(y).
	\end{aligned}
\end{equation*}
The bad cubes \(Q\) are pairwise disjoint. Moreover, the
Calder\'on-Zygmund decomposition at height \(1\) gives
\(\|h_Q\|_{L^1}\lesssim |Q|\). Consequently,
\begin{equation}
	\label{eq:lemma42-W-L3-bound}
	\|W\|_{L^3}\lesssim \Big(\sum_{J,Q}\|h_Q\|_{L^1}\Big)^{1/3}.
\end{equation}

Let \(M\) be the uncentered Hardy-Littlewood maximal operator and set
\[
M_{6/5}g
\coloneqq
\bigl(M(|g|^{6/5})\bigr)^{5/6}.
\]
Since \(Q\subset J\), for every \(y\in Q\) we have
\begin{equation}
	\label{eq:lemma42-gleqM6/5g}
	\bigg(\frac{1}{|J|}\int_J |g(x)|^{6/5}\,dx \bigg)^{5/6}
	\leq M_{6/5}g(y).
\end{equation}

Now take \(g\geq0\) with \(\|g\|_{L^{3/2}}=1\). By the
John-Nirenberg inequality,
\[
\begin{aligned}
	\left(\frac{1}{|J|}\int_J|b-b_Q|^6\,dx\right)^{1/6}
	&\leq
	\left(\frac{1}{|J|}\int_J|b-b_J|^6\,dx\right)^{1/6}
	+|b_J-b_Q|  \\
	&\lesssim s\|b\|_{\operatorname{BMO}(\mathbb R^d)}.
\end{aligned}
\]
The kernel size estimate and its support property give
\[
|T_jh_Q(x)|
\lesssim
\|\Omega\|_{L^\infty(\mathbb S^{d-1})}
\frac{\|h_Q\|_{L^1}}{|J|}\chi_J(x).
\]
Thus H\"older's inequality and \eqref{eq:lemma42-gleqM6/5g} give
\[
\begin{aligned}
	\int_{\mathbb R^d}|\mathcal E_{s,n,\iota}(x)|g(x)\,dx
	&\lesssim
	\|\Omega\|_{L^\infty}
	\sum_{J,Q}
	\|h_Q\|_{L^1}\frac{1}{|J|}
	\int_J|b(x)-b_Q|g(x)\,dx\\
	&\lesssim
	s\|\Omega\|_{L^\infty}
	\|b\|_{\operatorname{BMO}}
	\sum_{J,Q}
	\|h_Q\|_{L^1}
	\left(\frac{1}{|J|}\int_Jg(x)^{6/5}\,dx\right)^{5/6}\\
	&\leq
	s\|\Omega\|_{L^\infty}
	\|b\|_{\operatorname{BMO}}
	\int_{\mathbb R^d}W(y)M_{6/5}g(y)\,dy.
\end{aligned}
\]
By \eqref{eq:lemma42-W-L3-bound} and \(L^{5/4}\) boundedness of \(M\),
\[
\begin{aligned}
	\int_{\mathbb R^d}|\mathcal E_{s,n,\iota}(x)|g(x)\,dx & \leq
	s\|\Omega\|_{L^\infty}
	\|b\|_{\operatorname{BMO}} \|W\|_{L^3} \|M_{6/5}g\|_{L^{3/2}} \\
	&\lesssim
	s\|\Omega\|_{L^\infty}
	\|b\|_{\operatorname{BMO}}
	\Big(\sum_{J,Q} \|h_Q\|_{L^1}\Big)^{1/3} \|g\|_{L^{3/2}}.
\end{aligned}
\]
Taking the supremum over such \(g\) and applying
Lemma \ref{lem:layered-bad-mass}, we obtain
\begin{equation}
	\label{eq:lemma42-L3}
	\|\mathcal E_{s,n,\iota}\|_{L^3(\mathbb R^d)}
	\lesssim
	s\,2^{-2n/3}
	\|\Omega\|_{L^\infty(\mathbb S^{d-1})}
	\|b\|_{\operatorname{BMO}(\mathbb R^d)}
	\|f\|_{L^1(\mathbb R^d)}^{1/3}.
\end{equation}

\medskip
\noindent
\textbf{\(L^1\) estimate.}
Fix \(J\in\mathcal I_{s,\iota}^{\#,n}\) and
\(Q\in\mathfrak{Q}_{J-s}\). The definition
\(\widetilde\chi_{J^\iota}=P_{j-s\kappa}\chi_{J^\iota}\) implies that
\begin{equation*}
	\operatorname{supp}
	\big(\chi_{J^\iota}-\widetilde\chi_{J^\iota}\big)
	\subset E_J  \coloneqq  \big\{x: \operatorname{dist}(x,\partial J^\iota) \lesssim 2^{j-s\kappa}\big\}.
\end{equation*}
A comparison of scales shows that \(E_J\subset4J\) and
\begin{equation*}
	|E_J|
	\lesssim 2^{jd-s\kappa}
	=2^{-s\kappa}|J|.
	\label{eq:lemma42-EJ-measure}
\end{equation*}
It follows that
\begin{equation}
	\label{eq:lemma42-single-L1}
	\begin{aligned}
		&\big\|
		\big(b-b_Q\big)T_jh_Q
		\big(\chi_{J^\iota}-\widetilde\chi_{J^\iota}\big)
		\big\|_{L^1(\mathbb R^d)} \\
		& \quad \lesssim
		2^{-jd}
		\|\Omega\|_{L^\infty(\mathbb S^{d-1})}
		\|h_Q\|_{L^1(\mathbb R^d)}
		\int_{E_J}|b(x)-b_Q|dx.
	\end{aligned}
\end{equation}
Applying Corollary \ref{cor:bmo-subset-integral} with \(R=4J\) and \(E=E_J\), we obtain
\begin{equation*}
	\int_{E_J}|b(x)-b_{4J}|dx
	\lesssim (1+s\kappa)2^{jd-s\kappa}
	\|b\|_{\operatorname{BMO}(\mathbb R^d)}.
\end{equation*}
On the other hand, since
\(Q\subset J\subset4J\) and
\(\ell(4J)/\ell(Q)=2^{s+2}\), we have
\begin{equation*}
	|b_{4J}-b_Q|
	\lesssim
	s\|b\|_{\operatorname{BMO}(\mathbb R^d)}.
\end{equation*}
Therefore,
\begin{equation}
	\label{eq:lemma42-EJ-bQ}
	\begin{aligned}
		\int_{E_J}|b(x)-b_Q|dx
		&\leq
		\int_{E_J}|b(x)-b_{4J}|dx
		+|b_{4J}-b_Q||E_J| \\
		&\lesssim
		s2^{jd-s\kappa}
		\|b\|_{\operatorname{BMO}(\mathbb R^d)}.
	\end{aligned}
\end{equation}
Substituting \eqref{eq:lemma42-EJ-bQ} into
\eqref{eq:lemma42-single-L1}, we obtain
\begin{align*}
	&\big\|
	\big(b-b_Q\big)T_jh_Q
	\big(\chi_{J^\iota}-\widetilde\chi_{J^\iota}\big)
	\big\|_{L^1(\mathbb R^d)}\\
	&\quad\lesssim
	s2^{-s\kappa}
	\|\Omega\|_{L^\infty(\mathbb S^{d-1})}
	\|b\|_{\operatorname{BMO}(\mathbb R^d)}
	\|h_Q\|_{L^1(\mathbb R^d)}.
\end{align*}
Summing over \(J\) and \(Q\) and applying Lemma \ref{lem:layered-bad-mass} gives
\begin{equation}		\label{eq:lemma42-L1} 
	\begin{aligned}
		\| \mathcal E_{s,n,\iota} \|_{L^1(\mathbb R^d)} \lesssim 2^{-2n}
		s2^{-s\kappa}
		\|\Omega\|_{L^\infty(\mathbb S^{d-1})}
		\|b\|_{\operatorname{BMO}(\mathbb R^d)}
		\|f\|_{L^1}.
	\end{aligned}
\end{equation}

\noindent
\textbf{Interpolation.}
Interpolating between \eqref{eq:lemma42-L1} and \eqref{eq:lemma42-L3}, we obtain
\begin{align*}
	\|\mathcal E_{s,n,\iota}\|_{L^2(\mathbb R^d)}
	&\leq
	\|\mathcal E_{s,n,\iota}\|_{L^1(\mathbb R^d)}^{1/4}
	\|\mathcal E_{s,n,\iota}\|_{L^3(\mathbb R^d)}^{3/4} \\
	& \lesssim 2^{-n} s2^{-s\kappa/4}
	\|\Omega\|_{L^\infty(\mathbb S^{d-1})}
	\|b\|_{\operatorname{BMO}(\mathbb R^d)}
	\|f\|_{L^1}^{1/2}.
\end{align*}
The bounded kernel considered here satisfies
\begin{equation*}
	\|\Omega\|_{L^\infty(\mathbb S^{d-1})} \leq 2^{\eta s}.
\end{equation*}
Consequently,
\begin{align*}
	\|\mathcal E_{s,n,\iota}\|_{L^2(\mathbb R^d)} \lesssim 2^{-s(\kappa/4-\eta)}2^{-n}
	\|b\|_{\operatorname{BMO}(\mathbb R^d)}
	\|f\|_{L^1(\mathbb R^d)}^{1/2}.
\end{align*}
Choose $\eta<\kappa/4$.  The polynomial factor $s$ is then absorbed by a
slight decrease of the decay exponent, which proves Lemma
\ref{lem:sec42-boundary-error}.

We next prove Lemma \ref{lem:sec42-low-frequency}.  Set
\begin{align*}
	\mathcal E_{s,n,\iota}(x)
	\coloneqq  \sum_{J\in\mathcal I_{s,\iota}^{\#,n}}
	\sum_{Q\in\mathfrak Q_{J-s}}
	\big(b(x)-b_Q\big)
	P_{j-s\kappa}\Big[
	\varepsilon_J\big(T_jh_Q\big)
	\widetilde\chi_{J^\iota}
	\Big](x).
\end{align*}

\medskip
\noindent
\textbf{$L^3$ estimate.}
Since $0\leq\widetilde\chi_{J^\iota}\leq1$,
$|\varepsilon_J|=1$, and $P_{j-s\kappa}$ is convolution with the nonnegative function
$\varpi_{j-s\kappa}$, we have
\begin{equation*}
	\begin{aligned}
		&\big|
		P_{j-s\kappa}\big[
		\varepsilon_J (T_jh_Q)
		\widetilde\chi_{J^\iota}
		\big](x)
		\big|
		\\
		&\qquad\leq
		P_{j-s\kappa}\big[|T_jh_Q|\big](x)
		\\
		&\qquad\leq
		\big(
		\varpi_{j-s\kappa}*|\mathcal K_j|
		\big)*|h_Q|(x)
		\lesssim
		\mathcal H_j*|h_Q|(x),
	\end{aligned}
\end{equation*}
where
\begin{equation*}
	\mathcal H_j(x)
	\coloneqq 
	2^{-jd}
	\|\Omega\|_{L^\infty(\mathbb S^{d-1})}
	\chi_{\{2^{j-5}<|x|<2^{j-1}\}}(x).
\end{equation*}
This follows from
\[
\operatorname{supp}\varpi_{j-s\kappa}
\subset
\big\{x:|x|\leq2^{j-s\kappa-5}\big\},
\]
together with the support and size estimates for $\mathcal K_j$. Thus $\mathcal H_j$ has the same scale, support, and size properties as
$|\mathcal K_j|$ in the $L^3$ estimate of Lemma \ref{lem:sec42-boundary-error}. Repeating the argument used in Lemma \ref{lem:sec42-boundary-error} gives
\begin{equation}\label{eq:lemma43-L3}
	\begin{aligned}
		\|\mathcal E_{s,n,\iota}\|_{L^3(\mathbb R^d)}
		&\lesssim
		s\|\Omega\|_{L^\infty(\mathbb S^{d-1})}
		\|b\|_{\operatorname{BMO}(\mathbb R^d)}
		\bigg( \sum_{ J\in\mathcal I_{s,\iota}^{\#,n}} \sum_{Q\in\mathfrak Q_{J-s}}
		\|h_Q\|_{L^1(\mathbb R^d)}
		\bigg)^{1/3}
		\\
		&\lesssim
		s2^{-2n/3}
		\|\Omega\|_{L^\infty(\mathbb S^{d-1})}
		\|b\|_{\operatorname{BMO}(\mathbb R^d)}
		\|f\|_{L^1(\mathbb R^d)}^{1/3},
	\end{aligned}
\end{equation}
where the last step follows from Lemma
\ref{lem:layered-bad-mass}.

\medskip
\noindent
\textbf{$L^1$ estimate.}
Fix $J\in\mathcal I_{s,\iota}^{\#,n}$ and
$Q\in\mathfrak Q_{J-s}$. Let $y_Q$ be the center of $Q$. For fixed $x$, define
\begin{equation*}
	\begin{aligned}
		\mathcal K_x(y)
		& \coloneqq \int_{\mathbb R^d}
		\varpi_{j-s\kappa}(x-w)
		\mathcal K_j(w-y)
		\widetilde\chi_{J^\iota}(w)dw\\
		&=\int_{\mathbb R^d}
		\varpi_{j-s\kappa}(x-y-z)
		\mathcal K_j(z)
		\widetilde\chi_{J^\iota}(y+z)dz.
	\end{aligned}
\end{equation*}
Then
\begin{equation*}
	P_{j-s\kappa}\big[
	\varepsilon_J\big(T_jh_Q\big)
	\widetilde\chi_{J^\iota}
	\big](x)
	=\varepsilon_J\int_Qh_Q(y)\mathcal K_x(y)dy.
\end{equation*}
Since $h_Q$ has mean zero,
\begin{equation}
	\label{eq:lemma43-use-cancellation}
	\begin{aligned}
		&\big(b(x)-b_Q\big)
		P_{j-s\kappa}\big[
		\varepsilon_J\big(T_jh_Q\big)
		\widetilde\chi_{J^\iota}
		\big](x) \\
		&\qquad=
		\varepsilon_J\big(b(x)-b_Q\big)
		\int_Qh_Q(y)
		\big(\mathcal K_x(y)-\mathcal K_x(y_Q)\big)dy.
	\end{aligned}
\end{equation}
For $0\leq t\leq1$, set
\begin{equation*}
	y_t \coloneqq ty+(1-t)y_Q.
\end{equation*}
By the fundamental theorem of calculus,
\begin{equation}
	\mathcal K_x(y)-\mathcal K_x(y_Q)
	=\int_0^1
	\big\langle y-y_Q,\nabla_y\mathcal K_x(y_t)\big\rangle dt.
	\label{eq:lemma43-mean-value}
\end{equation}
Differentiating $\mathcal K_x(y)$ with respect to $y$ gives
\begin{equation}
	\label{eq:lemma43-kernel-derivative}
	\begin{aligned}
		\nabla_y\mathcal K_x(y_t)
		&=-\int_{\mathbb R^d}
		\nabla\varpi_{j-s\kappa}(x-y_t-z)
		\mathcal K_j(z)
		\widetilde\chi_{J^\iota}(y_t+z)dz \\
		&\quad+
		\int_{\mathbb R^d}
		\varpi_{j-s\kappa}(x-y_t-z)
		\mathcal K_j(z)
		\nabla\widetilde\chi_{J^\iota}(y_t+z)dz.
	\end{aligned}
\end{equation}
We next integrate \eqref{eq:lemma43-kernel-derivative} against the BMO
oscillation in $x$.
Since $y_t\in Q\subset\frac12J$, $\operatorname{supp}\mathcal K_j
\subset\{z:|z|\leq 2^{j-2}\}$, and
$\operatorname{supp}\varpi_{j-s\kappa}
\subset\{x:|x| \leq 2^{j-s\kappa-5}\}$, every $x$ occurring in the integral
lies in $ 16J $. We use the following two estimates for $\BMO$ functions:
\begin{align}
	&\int_{\mathbb R^d}
	|b(x)-b_Q|
	\varpi_{j-s\kappa}(x-a)dx
	\lesssim
	s\|b\|_{\operatorname{BMO}(\mathbb R^d)},
	\label{eq:lemma43-BMO-varpi}\\
	&\int_{\mathbb R^d}
	|b(x)-b_Q|
	|\nabla\varpi_{j-s\kappa}(x-a)|dx
	\lesssim
	s2^{-j+s\kappa}
	\|b\|_{\operatorname{BMO}(\mathbb R^d)},
	\label{eq:lemma43-BMO-gradient-varpi}
\end{align}
Here $a=y_t+z$, with $z$ restricted to the support of $\mathcal K_j$.
To see these bounds, choose a cube $R_a$ containing
$a+\operatorname{supp}\varpi_{j-s\kappa}$ with
$\ell(R_a)\sim2^{j-s\kappa}$. The support conditions allow us to take
$R_a\subset 16J $. We then have
\begin{align*}
	\frac{1}{|R_a|}\int_{R_a}|b(x)-b_Q|dx
	&\leq
	\frac{1}{|R_a|}\int_{R_a}|b(x)-b_{R_a}|dx
	+|b_{R_a}-b_{ 16J }|+|b_{ 16J }-b_Q|\\
	&\lesssim
	s\|b\|_{\operatorname{BMO}(\mathbb R^d)}.
\end{align*}
The last inequality follows by comparing BMO averages across the intervening
scales.
Also,
\begin{align*}
	\varpi_{j-s\kappa}(x-a) & \lesssim\frac{\chi_{R_a}(x)}{|R_a|}, \\
	|\nabla\varpi_{j-s\kappa}(x-a)| & \lesssim
	2^{-j+s\kappa}\frac{\chi_{R_a}(x)}{|R_a|}.
\end{align*}
This proves
\eqref{eq:lemma43-BMO-varpi} and \eqref{eq:lemma43-BMO-gradient-varpi}.
On the other hand,
\begin{align}
	\label{eq:lemma43-kernel-bound}
	& \|\mathcal K_j\|_{L^1(\mathbb R^d)} \lesssim\|\Omega\|_{L^\infty(\mathbb S^{d-1})}, \\
	& \label{eq:lemma43-cutoff-bound}
	\|\nabla\widetilde\chi_{J^\iota}\|_{L^\infty(\mathbb R^d)}
	\lesssim2^{-j+s\kappa}.
\end{align}
Substituting \eqref{eq:lemma43-BMO-varpi}-\eqref{eq:lemma43-cutoff-bound} into
\eqref{eq:lemma43-kernel-derivative}, integrating in $x$, and applying Fubini's theorem, we obtain
\begin{equation}		\label{eq:lemma43-weighted-kernel-derivative}
	\begin{aligned}
		& \int_{\mathbb R^d}
		|b(x)-b_Q|
		|\nabla_y\mathcal K_x(y_t)|dx \\
		&\qquad\lesssim
		s2^{-j+s\kappa}
		\|\Omega\|_{L^\infty(\mathbb S^{d-1})}
		\|b\|_{\operatorname{BMO}(\mathbb R^d)}.
	\end{aligned}
\end{equation}
Since $y,y_Q\in Q$,
\begin{equation}
	|y-y_Q|\lesssim2^{j-s}.
	\label{eq:lemma43-y-distance}
\end{equation}
Combining \eqref{eq:lemma43-use-cancellation},
\eqref{eq:lemma43-mean-value},
\eqref{eq:lemma43-weighted-kernel-derivative}, and
\eqref{eq:lemma43-y-distance}, we obtain
\begin{equation*}\label{eq:lemma43-single-L1}
	\begin{aligned}
		&\Big\|
		\big(b-b_Q\big)
		P_{j-s\kappa}\big[
		\varepsilon_J\big(T_jh_Q\big)
		\widetilde\chi_{J^\iota}
		\big]
		\Big\|_{L^1(\mathbb R^d)} \\
		&\qquad\lesssim
		s2^{-s(1-\kappa)}
		\|\Omega\|_{L^\infty(\mathbb S^{d-1})}
		\|b\|_{\operatorname{BMO}(\mathbb R^d)}
		\|h_Q\|_{L^1(\mathbb R^d)}.
	\end{aligned}
\end{equation*}
Summing over $J$ and $Q$ gives
\begin{equation}
	\|\mathcal E_{s,n,\iota}\|_{L^1(\mathbb R^d)}
	\lesssim
	s2^{-s(1-\kappa)} 2^{-2n} \|\Omega\|_{L^\infty(\mathbb S^{d-1})}
	\|b\|_{\operatorname{BMO}(\mathbb R^d)} \|f\|_{L^1}.
	\label{eq:lemma43-L1}
\end{equation}

\medskip
\noindent
\textbf{Interpolation.}
Interpolating between \eqref{eq:lemma43-L1} and \eqref{eq:lemma43-L3},
we obtain
\begin{equation*}
	\begin{aligned}
		\|\mathcal E_{s,n,\iota}\|_{L^2(\mathbb R^d)}
		&\leq
		\|\mathcal E_{s,n,\iota}\|_{L^1(\mathbb R^d)}^{1/4}
		\|\mathcal E_{s,n,\iota}\|_{L^3(\mathbb R^d)}^{3/4}
		\notag\\
		&\lesssim
		s2^{-s(1-\kappa)/4} 2^{-n}
		\|\Omega\|_{L^\infty(\mathbb S^{d-1})}
		\|b\|_{\operatorname{BMO}(\mathbb R^d)}\|f\|_{L^1}^{1/2}.
		\label{eq:lemma43-interpolation}
	\end{aligned}
\end{equation*}
Using $\|\Omega\|_{L^\infty(\mathbb S^{d-1})} \lesssim 2^{s\eta}$, we deduce
\begin{equation*}
	\|\mathcal E_{s,n,\iota}\|_{L^2(\mathbb R^d)} \lesssim 2^{-s\big( \frac{1-\kappa}{4}-\eta\big)}2^{-n}
	\|b\|_{\operatorname{BMO}(\mathbb R^d)} \|f\|_{L^1(\mathbb R^d)}^{1/2}.
\end{equation*}
It suffices to choose
\begin{equation*}
	0<\kappa<1,
	\qquad
	0<\eta<\frac{1-\kappa}{4},
\end{equation*}
and then absorb the factor $s$ by decreasing the decay exponent.  This proves
Lemma \ref{lem:sec42-low-frequency}.
\vspace{0.3cm}

\section{Proof of Lemmas \ref{lem:sec42-main-microlocal} and  \ref{lem:sec42-commutator-microlocal}}
\label{sec:proof-lemma45}
\begin{proof} [{ Proof of Lemma \ref{lem:sec42-main-microlocal}}]

Set
\begin{equation*}
	\mathcal E_{s,n,\iota}
	\coloneqq 
	\sum_{J\in\mathcal I_{s,\iota}^{\#,n}}
	\sum_{Q\in\mathfrak Q_{J-s}}
	\sum_{\nu}
	G_\nu \Big[
	\big(b-b_Q\big)
	\big(I-P_{j-s\kappa}\big)\Big[
	\varepsilon_J\big(T_{j,\nu} h_Q\big)
	\widetilde\chi_{J^\iota}
	\Big]
	\Big].
\end{equation*}
We prove the required decay by a $TT^*$ argument.  First decompose $\Esn$
into directionally localized terms $\Esn^\nu$.

Recall that
\begin{equation*}
	\widehat{G_\nu  u}(\xi)
	=\Psi\Big(
	2^{s\gamma}
	\Big\langle e_\nu,\frac{\xi}{|\xi|}\Big\rangle
	\Big)\widehat u(\xi).
\end{equation*}
The support of $\Psi$ implies that
\begin{equation*}
	\operatorname{dist}\Big(
	e_\nu,\big(\xi/|\xi|\big)^\perp
	\Big)
	\lesssim 2^{-s\gamma}.
\end{equation*}
Since the directions $\{e_\nu\}_\nu$ are $2^{-s\gamma-4}$-separated, this
spherical band contains at most $C_d2^{s\gamma(d-2)}$ such points.  Hence
\begin{equation*}
	\sup_{\xi\neq0}
	\sum_{\nu}
	\Big|
	\Psi\Big(
	2^{s\gamma}
	\Big\langle e_\nu,\frac{\xi}{|\xi|}\Big\rangle
	\Big)
	\Big|^2
	\lesssim2^{s\gamma(d-2)}.
\end{equation*}
Plancherel's theorem yields
\begin{equation}\label{eq:lemma44-G-square-function}
	\sum_{\nu}
	\|G_\nu u\|_{L^2(\mathbb R^d)}^2
	\lesssim
	2^{s\gamma(d-2)}
	\|u\|_{L^2(\mathbb R^d)}^2.
\end{equation}
For each fixed $\nu$, set
\begin{equation*}
	\mathcal E_{s,n,\iota}^\nu
	\coloneqq 
	\sum_{J\in\mathcal I_{s,\iota}^{\#,n}}
	\sum_{Q\in\mathfrak Q_{J-s}}
	\big(b-b_Q\big)
	\big(I-P_{j-s\kappa}\big)\Big[
	\varepsilon_J\big(T_{j,\nu} h_Q\big)
	\widetilde\chi_{J^\iota}
	\Big].
\end{equation*}
For any $u \in L^2(\mathbb R^d)$, the Cauchy-Schwarz
inequality and \eqref{eq:lemma44-G-square-function} give
\begin{equation*}
	\begin{aligned}
		\big|\langle\mathcal E_{s,n,\iota},u \rangle\big|
		&=\Big|
		\sum_{\nu}
		\langle\mathcal E_{s,n,\iota}^\nu,G_\nu  u \rangle \Big|\\
		&\leq \Big( \sum_{\nu}
		\|\mathcal E_{s,n,\iota}^\nu\|_{L^2(\mathbb R^d)}^2
		\Big)^{1/2}
		\Big(
		\sum_{\nu}
		\|G_\nu  u\|_{L^2(\mathbb R^d)}^2
		\Big)^{1/2}\\
		&\lesssim
		2^{s\gamma(d-2)/2}
		\Big(
		\sum_{\nu}
		\|\mathcal E_{s,n,\iota}^\nu\|_{L^2(\mathbb R^d)}^2
		\Big)^{1/2}
		\|u\|_{L^2(\mathbb R^d)}.
	\end{aligned}
\end{equation*}
Therefore,
\begin{equation}\label{eq:lemma44-reduce-to-fixed-directions}
	\|\mathcal E_{s,n,\iota}\|_{L^2(\mathbb R^d)}^2
	\lesssim
	2^{s\gamma(d-2)}
	\sum_{\nu}
	\|\mathcal E_{s,n,\iota}^\nu\|_{L^2(\mathbb R^d)}^2.
\end{equation}

Fix a direction $\nu$ and define
\begin{equation*}
	E_{j,\nu}  
	\coloneqq \Big\{
	x \in \mathbb R^d: |\langle x,e_\nu\rangle|\leq2^{j},
	\quad
	\big|x-\langle x,e_\nu\rangle e_\nu\big|
	\leq 2^{j-s\gamma}
	\Big\}.
\end{equation*}
Its measure satisfies
\begin{equation}\label{eq:lemma44-sector-measure}
	|  E_{j,\nu}  |
	\lesssim2^{jd-s\gamma(d-1)}.
\end{equation}
For $Q\in\mathfrak Q_{J-s}$, define the kernel
\begin{equation}\label{eq:lemma44-composed-kernel}
	\begin{aligned}
		L_{J,Q}^\nu(x,y)
		& \coloneqq \varepsilon_J\mathcal K_{j,\nu}(x-y)
		\widetilde\chi_{J^\iota}(x)\\
		&\quad-
		\varepsilon_J
		\int_{\mathbb R^d}
		\varpi_{j-s\kappa}(x-z)
		\mathcal K_{j,\nu}(z-y)
		\widetilde\chi_{J^\iota}(z)dz.
	\end{aligned}
\end{equation}
Then
\begin{equation*}
	\big(I-P_{j-s\kappa}\big)\Big[
	\varepsilon_J\big(T_{j,\nu} h_Q\big)
	\widetilde\chi_{J^\iota}
	\Big](x)
	=\int_Q L_{J,Q}^\nu(x,y)h_Q(y)dy.
\end{equation*}
Hence,
\begin{equation}\label{eq:lemma44-fixed-direction-kernel-form}
	\mathcal E_{s,n,\iota}^\nu(x)
	=
	\sum_{J\in\mathcal I_{s,\iota}^{\#,n}}
	\sum_{Q\in\mathfrak Q_{J-s}}
	\big(b(x)-b_Q\big)
	\int_Q L_{J,Q}^\nu(x,y)h_Q(y)dy.
\end{equation}
The support properties of $\mathcal K_{j,\nu}$ and $\varpi_{j-s\kappa}$,
together with $0<\gamma\ll\kappa<1$, give
\begin{equation}\label{eq:lemma44-composed-kernel-bound}
	| L_{J,Q}^\nu(x,y)|
	\lesssim
	2^{-jd}\|\Omega\|_{L^\infty(\mathbb S^{d-1})}
	\chi_{  E_{j,\nu}  }(x-y).
\end{equation}

To estimate $\|\mathcal E_{s,n,\iota}^\nu\|_2^2$, let $I$ denote a second
localization cube, with $\ell(I)=2^i$.  Expanding the square in
\eqref{eq:lemma44-fixed-direction-kernel-form} gives
\begin{equation}\label{eq:lemma44-full-expansion}
	\begin{aligned}
		\|\mathcal E_{s,n,\iota}^\nu\|_{L^2(\mathbb R^d)}^2 = & \sum_{J\in\mathcal I_{s,\iota}^{\#,n}}
		\sum_{Q\in\mathfrak Q_{J-s}}
		\sum_{I\in\mathcal I_{s,\iota}^{\#,n}}
		\sum_{R\in\mathfrak Q_{I-s}}
		\int_Q\int_R h_Q(y)\overline{h_R(z)}
		\\
		&\quad\times
		\int_{\mathbb R^d}
		\big(b(x)-b_Q\big)\overline{\big(b(x)-b_R\big)}
		L_{J,Q}^\nu(x,y)
		\overline{L_{I,R}^\nu(x,z)}\,\dd x\,\dd z\,\dd y.
	\end{aligned}
\end{equation}
The sum is symmetric in $(J,Q)$ and $(I,R)$.  After taking absolute values,
we may therefore restrict to $\ell(I)\leq\ell(J)$.  The support of
$L_{J,Q}^\nu(x,y)$ then gives
\begin{equation}\label{eq:lemma44-TTstar-reduction}
	\begin{aligned}
		\|\mathcal E_{s,n,\iota}^\nu\|_{L^2(\mathbb R^d)}^2
		\lesssim{}&
		\sum_{J\in\mathcal I_{s,\iota}^{\#,n}}
		\sum_{Q\in\mathfrak Q_{J-s}}
		\int_Q|h_Q(y)|\\
		&\quad\times
		\int_{\mathbb R^d}
		| L_{J,Q}^\nu(x,y)|
		|b(x)-b_Q|
		S_J^\nu(x,y)dxdy,
	\end{aligned}
\end{equation}
where
\begin{equation}\label{eq:lemma44-SJ-definition}
	\begin{aligned}
		S_J^\nu(x,y)
		\coloneqq &\ 
		\chi_{y+  E_{j,\nu}  }(x)
		\sum_{\substack{I\in\mathcal I_{s,\iota}^{\#,n}\\
				\ell(I)\leq\ell(J)}}
		\sum_{R\in\mathfrak Q_{I-s} }
		\big| b(x)-b_R \big|\\
		&\quad\times
		\int_R |  L_{I,R}^\nu(x,z)| \, |h_R(z)|dz.
	\end{aligned}
\end{equation}
A comparison of supports shows that $R$ meets
$y+E_{j,\nu}-E_{i,\nu}$.  Since $i\leq j$, $\ell(R)=2^{i-s}$, and
$0<\gamma<1$, every such $R$ is contained in $y+2E_{j,\nu}$.

We next estimate the local $L^1$ norm of $S_J^\nu(\cdot,y)$. The definition of $ E_{j,\nu} $ and the bound
$\|h_R\|_1\lesssim|R|$ from the Calder\'on-Zygmund decomposition imply
\begin{equation}\label{eq:lemma44-local-bad-mass}
	\sum_{\substack{I\in\mathcal I_{s,\iota}^{\#,n}\\
			\ell(I)\leq\ell(J)}}
	\sum_{\substack{R\in\mathfrak Q_{I-s}\\
			R\cap(y+2  E_{j,\nu}  )\neq\varnothing}}
	\|h_R\|_{L^1(\mathbb R^d)}
	\lesssim
	2^{jd-s\gamma(d-1)}.
\end{equation}

Fix such $I,R$ and $z\in R$. Set $R^\ast=2^{s+2} R$; then $z+ E_{i,\nu} \subset R^\ast$. Note that
\begin{equation*}
	|z+ E_{i,\nu} |
	\lesssim2^{id-s\gamma(d-1)}.
\end{equation*}
By Corollary \ref{cor:bmo-subset-integral},
\begin{equation}\label{eq:lemma44-bmo-on-sector}
	\begin{aligned}
		\int_{z+ E_{i,\nu} }
		|b(x)-b_R|dx
		&\leq
		\int_{z+ E_{i,\nu} }
		|b(x)-b_{R^\ast}|dx\\
		&\quad+
		|z+ E_{i,\nu} |
		|b_{R^\ast}-b_R|\\
		&\lesssim
		s\|b\|_{\operatorname{BMO}(\mathbb R^d)}
		2^{id-s\gamma(d-1)}.
	\end{aligned}
\end{equation}
Substituting \eqref{eq:lemma44-composed-kernel-bound},
\eqref{eq:lemma44-local-bad-mass}, and
\eqref{eq:lemma44-bmo-on-sector} into
\eqref{eq:lemma44-SJ-definition}, we obtain
\begin{equation}\label{eq:lemma44-SJ-L1}
	\begin{aligned}
		\|S_J^\nu(\cdot,y)\|_{L^1( 16J )}
		&\lesssim
		s\|\Omega\|_{L^\infty(\mathbb S^{d-1})}
		\|b\|_{\operatorname{BMO}(\mathbb R^d)}
		2^{-s\gamma(d-1)}\\
		&\quad\times
		\sum_{\substack{I\in\mathcal I_{s,\iota}^{\#,n}\\
				\ell(I)\leq\ell(J)}}
		\sum_{\substack{R\in\mathfrak Q_{I-s}\\
				R\cap(y+2  E_{j,\nu}  )\neq\varnothing}}
		\|h_R\|_{L^1(\mathbb R^d)}\\
		&\lesssim
		s\|\Omega\|_{L^\infty(\mathbb S^{d-1})}
		\|b\|_{\operatorname{BMO}(\mathbb R^d)}
		2^{jd-2s\gamma(d-1)}.
	\end{aligned}
\end{equation}

We also require the following uniform Orlicz bound for
$S_J^\nu(\cdot,y)$:
\begin{equation}\label{eq:lemma44-SJ-orlicz}
	\|S_J^\nu(\cdot,y)\|_{L^2(\log L)^{-2}, 16J }
	\lesssim
	s\|\Omega\|_{L^\infty(\mathbb S^{d-1})}
	\|b\|_{\operatorname{BMO}(\mathbb R^d)}.
\end{equation}
By Orlicz duality, it suffices to consider a function $g$ supported in $16J$
such that
\begin{equation*}
	\|g\|_{L^2(\log L)^2, 16J }\leq1
\end{equation*}
and estimate
$| 16J |^{-1}\int_{ 16J }|S_J^\nu(x,y)g(x)|dx$.
Let $M$ be the Hardy-Littlewood maximal operator. By the John-Nirenberg theorem,
$|b_{R^\ast}-b_R|\lesssim s\|b\|_{\operatorname{BMO}}$, and
Young's inequality
\begin{equation*}
	uv\lesssim e^u-1+v\log(e+v),
	\qquad u,v\geq0,
\end{equation*}
we have, for every $z\in R$,
\begin{equation}\label{eq:lemma44-young-maximal-bound}
	\begin{aligned}
		&2^{-id}
		\int_{(z+E_{i,\nu})\cap 16J}
		|b(x)-b_R||g(x)|dx\\
		&\quad\lesssim
		s\|b\|_{\operatorname{BMO}(\mathbb R^d)}
		2^{-id}
		\int_{R^\ast}
		\Bigg[
		\exp\Big(
		\frac{|b(x)-b_R|}
		{C_ds\|b\|_{\operatorname{BMO}(\mathbb R^d)}}
		\Big)-1\\
		&\hspace{6.5cm}
		+|g(x)|\log\big(e+|g(x)|\big) \Bigg]dx. \\
		&\quad\lesssim
		s\|b\|_{\operatorname{BMO}(\mathbb R^d)}
		\Bigg[
		1+
		\inf_{x\in R}
		M\Big(
		|g|\log\big(e+|g|\big) \Big)(x)
		\Bigg].
	\end{aligned}
\end{equation}

Summing over $I,R$ and using \eqref{eq:lemma44-young-maximal-bound}, \eqref{eq:lemma44-composed-kernel-bound}, and the estimate for the bad functions in the Calder\'on-Zygmund decomposition, we obtain
\begin{equation}\label{eq:lemma44-orlicz-duality-estimate}
	\begin{aligned}
		&\frac{1}{| 16J |}
		\int_{ 16J }|S_J^\nu(x,y)g(x)|dx\\
		&\quad\lesssim
		\frac{s\|\Omega\|_{L^\infty(\mathbb S^{d-1})}
			\|b\|_{\operatorname{BMO}(\mathbb R^d)}}{| 16J |} \sum_I \sum_R \|h_R\|_{L^1}
		\Big[ 1 + \inf_{x\in R}
		M\big(|g|\log(e+|g|)\big)(x)
		\Big]\\
		&\quad\lesssim
		\frac{s\|\Omega\|_{L^\infty(\mathbb S^{d-1})}
			\|b\|_{\operatorname{BMO}(\mathbb R^d)}}{| 16J |} \sum_I \sum_R \Big[|R| + |R|
		\inf_{x \in R}
		M\big(|g|\log(e+|g|)\chi_{ 16J }\big)(x)
		\Big]\\
		&\quad\lesssim
		\frac{s\|\Omega\|_{L^\infty(\mathbb S^{d-1})}
			\|b\|_{\operatorname{BMO}(\mathbb R^d)}}{| 16J |}
		\Big[
		| 16J |+
		\int_{ 16J }M\big(|g|\log(e+|g|) \big)(u)du
		\Big].
	\end{aligned}
\end{equation}
All sums over $R$ in \eqref{eq:lemma44-orlicz-duality-estimate} are restricted
to the bad cubes occurring in \eqref{eq:lemma44-local-bad-mass}.  These cubes
are pairwise disjoint, so the terms containing the infimum are bounded by the
corresponding integral.  The $L^2$ boundedness of $M$ and the normalization of
$g$ then give
\begin{equation}\label{eq:lemma44-maximal-L2}
	\begin{aligned}
		\int_{ 16J }M\big(|g|\log(e+|g|)\big)(u)du
		&\leq
		| 16J |^{1/2}
		\|M\big(|g|\log(e+|g|) \big)\|_{L^2(\mathbb R^d)}\\
		&\lesssim
		| 16J |^{1/2}
		\|g\log(e+|g|)\|_{L^2( 16J )}\\
		&\lesssim| 16J |.
	\end{aligned}
\end{equation}
Together, \eqref{eq:lemma44-orlicz-duality-estimate} and
\eqref{eq:lemma44-maximal-L2} prove
\eqref{eq:lemma44-SJ-orlicz}.

We now apply Lemma \ref{lem:orlicz-local-interpolation} using
\eqref{eq:lemma44-SJ-L1} and \eqref{eq:lemma44-SJ-orlicz}.  Take
$p=2$, $\alpha=2$, and choose $1<q<2$, $0<r<1$, and
$0<\epsilon<r$ so that
\begin{equation*}
	\frac{1}{q}=r+\frac{1-r}{2}.
\end{equation*}
Applying the interpolation lemma to
\begin{equation*}
	\frac{S_J^\nu(\cdot,y)}{
		Cs\|\Omega\|_{L^\infty(\mathbb S^{d-1})}
		\|b\|_{\operatorname{BMO}(\mathbb R^d)}}
\end{equation*}
gives
\begin{equation}\label{eq:lemma44-SJ-Lq}
	\begin{aligned}
		2^{-jd/q}
		\|S_J^\nu(\cdot,y)\|_{L^q( 16J )}
		\lesssim{}&
		s\|\Omega\|_{L^\infty(\mathbb S^{d-1})}
		\|b\|_{\operatorname{BMO}(\mathbb R^d)}\\
		&\quad\times
		2^{-2\epsilon s\gamma(d-1)}.
	\end{aligned}
\end{equation}

We now return to the estimate of $\|\Esn^\nu\|_{L^2}$. The John-Nirenberg theorem and $\ell(J)/\ell(Q)=2^s$ give
\begin{equation}\label{eq:lemma44-bQ-Lqprime}
	\begin{aligned}
		\|b-b_Q\|_{L^{q'}( 16J )}
		&\leq
		\|b-b_{ 16J }\|_{L^{q'}( 16J )}
		+|b_{ 16J }-b_Q|| 16J |^{1/q'}\\
		&\lesssim_{q}
		s\|b\|_{\operatorname{BMO}(\mathbb R^d)}
		2^{jd/q'}.
	\end{aligned}
\end{equation}
By \eqref{eq:lemma44-composed-kernel-bound}, \eqref{eq:lemma44-SJ-Lq}, and
\eqref{eq:lemma44-bQ-Lqprime},
\begin{equation*}
	\begin{aligned}
		&\int_{\mathbb R^d}
		| L_{J,Q}^\nu(x,y)|
		|b(x)-b_Q||S_J^\nu(x,y)|dx\\
		&\quad\lesssim
		2^{-jd}\|\Omega\|_{L^\infty(\mathbb S^{d-1})}
		\|b-b_Q\|_{L^{q'}( 16J )}
		\|S_J^\nu(\cdot,y)\|_{L^q( 16J )}\\
		&\quad\lesssim_q
		s^2\|\Omega\|_{L^\infty(\mathbb S^{d-1})}^2
		\|b\|_{\operatorname{BMO}(\mathbb R^d)}^2
		2^{-2\epsilon s\gamma(d-1)}.
	\end{aligned}
\end{equation*}
Substituting this into \eqref{eq:lemma44-TTstar-reduction} yields
\begin{equation}\label{eq:lemma44-fixed-direction-L2}
	\|\mathcal E_{s,n,\iota}^\nu\|_{L^2(\mathbb R^d)}^2
	\lesssim_q
	s^2\|\Omega\|_{L^\infty(\mathbb S^{d-1})}^2
	\|b\|_{\operatorname{BMO}(\mathbb R^d)}^2
	2^{-2\epsilon s\gamma(d-1)}
	\sum_{J\in\mathcal I_{s,\iota}^{\#,n}}
	\sum_{Q\in\mathfrak Q_{J-s}}
	\|h_Q\|_{L^1(\mathbb R^d)}.
\end{equation}

Since $\card(\Theta_s)\lesssim2^{s\gamma(d-1)}$, combining
\eqref{eq:lemma44-reduce-to-fixed-directions} with
\eqref{eq:lemma44-fixed-direction-L2} gives
\begin{equation*}
	\begin{aligned}
		\|\mathcal E_{s,n,\iota}\|_{L^2(\mathbb R^d)}^2
		\lesssim_q{}&
		s^2\|\Omega\|_{L^\infty(\mathbb S^{d-1})}^2
		\|b\|_{\operatorname{BMO}(\mathbb R^d)}^2\\
		&\quad\times
		2^{s\gamma\{2d-3-2\epsilon(d-1)\}}
		\sum_{J\in\mathcal I_{s,\iota}^{\#,n}}
		\sum_{Q\in\mathfrak Q_{J-s}}
		\|h_Q\|_{L^1(\mathbb R^d)}.
	\end{aligned}
\end{equation*}
Choose $r$ sufficiently close to $1$, and then choose $\epsilon<r$
sufficiently close to $1$, so that
\begin{equation*}
	2\epsilon(d-1)>2d-\frac52.
\end{equation*}
Finally, the layered bad-function mass estimate Lemma \ref{lem:layered-bad-mass} and $\| \Omega \|_{L^\infty} \leq 2^{\eta s}$ yield
\begin{equation*}
	\|\mathcal E_{s,n,\iota}\|_{L^2(\mathbb R^d)}
	\lesssim
	2^{-n}2^{-s(\gamma/4 - \eta)} \|b\|_{\operatorname{BMO}(\mathbb R^d)}
	\|f\|_{L^1(\mathbb R^d)}^{1/2}.
\end{equation*}
Choosing $0<\eta<\gamma/4$ and absorbing the polynomial factor $s$ into the
exponential decay proves Lemma \ref{lem:sec42-main-microlocal}.

\end{proof}
\begin{proof} [{ Proof of Lemma \ref{lem:sec42-commutator-microlocal}}]
For $J\in\mathcal I_{s,\iota}^{\#,n}$ and
$Q\in\mathfrak Q_{J-s}$, define
\begin{equation*}
	u_{J,Q,\nu}
	\coloneqq 
	\big(I-P_{j-s\kappa}\big)\Big[
	\varepsilon_J\big(T_{j,\nu} h_Q\big)
	\widetilde\chi_{J^\iota}
	\Big]
\end{equation*}
and
\begin{equation*}
	u_\nu
	\coloneqq 
	\sum_{J\in\mathcal I_{s,\iota}^{\#,n}}
	\sum_{Q\in\mathfrak Q_{J-s}}
	u_{J,Q,\nu}.
\end{equation*}
The expression in Lemma \ref{lem:sec42-commutator-microlocal} is therefore
\begin{equation*}
	\mathcal E_{s,n,\iota}
	\coloneqq 
	\sum_{J\in\mathcal I_{s,\iota}^{\#,n}}
	\sum_{Q\in\mathfrak Q_{J-s}}
	\sum_\nu [b,G_\nu] u_{J,Q,\nu}.
\end{equation*}

Recall that the multiplier of $G_\nu $ is
\begin{equation*}
	m_\nu(\xi)
	\coloneqq 
	\Psi\left(
	2^{s\gamma}
	\left\langle e_\nu,\frac{\xi}{|\xi|}\right\rangle
	\right).
\end{equation*}
Thus $\widehat{G_\nu g}(\xi)=m_\nu(\xi)\widehat g(\xi)$.  We use the
analytic-family method of Coifman, Rochberg, and Weiss
\cite{CoifmanRochbergWeiss1976}, together with the quantitative weighted
bounds of Hyt\"onen and P\'erez \cite{HytonenPerez2013}, to prove the following
square-function estimate.
\begin{lemma}
	\label{lem:lemma45-commutator-micro}
	For every $u\in L^2(\Rd)$,
	\begin{equation*}
		\sum_\nu
		\|[b,G_\nu ]u\|_{L^2(\mathbb R^d)}^2
		\lesssim
		2^{s\gamma(d-\frac32)}
		\|b\|_{\operatorname{BMO}(\mathbb R^d)}^2
		\|u\|_{L^2(\mathbb R^d)}^2.
	\end{equation*}
\end{lemma}

\begin{proof}
	
	We first obtain a weighted $L^2$ square-function estimate for $G_\nu$ by
	Stein-Weiss interpolation.  The analytic-family argument then gives the
	corresponding estimate for the commutators.  We begin with the unweighted
	bound.
	
	If $m_\nu(\xi)\neq0$, then
	\begin{equation*}
		\left|
		\left\langle e_\nu,\frac{\xi}{|\xi|}\right\rangle
		\right|
		\lesssim2^{-s\gamma}.
	\end{equation*}
	Since the directions $\{e_\nu\}_{\nu}$ are
	$2^{-s\gamma-4}$-separated, for each fixed $\xi\neq0$ there are at most
	$C_d2^{s\gamma(d-2)}$ directions for which $m_\nu(\xi)\neq0$.
	Hence,
	\begin{equation*}
		\sup_{\xi\neq0}
		\sum_\nu|m_\nu(\xi)|^2
		\lesssim2^{s\gamma(d-2)}.
	\end{equation*}
	By Plancherel's theorem,
	\begin{equation}\label{eq:lemma45-unweighted-G-square-function}
		\sum_\nu
		\|G_\nu  u\|_{L^2(\mathbb R^d)}^2
		\lesssim
		2^{s\gamma(d-2)}
		\|u\|_{L^2(\mathbb R^d)}^2.
	\end{equation}
	
	For the weighted estimate, set
	\begin{equation*}
		N \coloneqq d.
	\end{equation*}
	By the chain rule, for every
	$|\alpha|\leq N$ we have
	\begin{equation*}
		|\partial_\xi^\alpha m_\nu(\xi)|
		\lesssim_{\alpha}
		2^{s\gamma|\alpha|}|\xi|^{-|\alpha|}
		\lesssim_{d}
		2^{s\gamma N}|\xi|^{-|\alpha|}.
	\end{equation*}
	Thus, for every weight $w\in A_2$, the weighted Mikhlin multiplier theorem gives
	\begin{equation*}
		\|G_\nu u\|_{L^2(w)}
		\lesssim_{[w]_{A_2}}
		2^{s\gamma N}\|u\|_{L^2(w)}.
	\end{equation*}
	Using also
	$ \card \{e_\nu\} \lesssim 2^{s\gamma(d-1)}$, we obtain
	\begin{equation}	\label{eq:lemma45-weighted-rough-square-function}
		\begin{aligned}
			\sum_\nu\|G_\nu u\|_{L^2(w)}^2
			&\lesssim_{[w]_{A_2}}
			2^{s\gamma(d-1+2N)}
			\|u\|_{L^2(w)}^2 \\
			&\lesssim_{[w]_{A_2}}
			2^{s\gamma(3d-1)}
			\|u\|_{L^2(w)}^2.
		\end{aligned}
	\end{equation}

	We interpolate between \eqref{eq:lemma45-unweighted-G-square-function} and
	\eqref{eq:lemma45-weighted-rough-square-function}.  Set
	\begin{equation*}
		a_0 \coloneqq 4d+4,
		\qquad
		\theta_0 \coloneqq \frac{1}{1+a_0}=\frac{1}{4d+5}.
	\end{equation*}
	Suppose that $w^{1+a_0}\in A_2$, and regard
	$g\mapsto\{G_\nu g\}_\nu$ as an operator with values in $\ell^2$.  Apply the
	Stein-Weiss theorem to the unweighted bound and the weighted bound with
	weight $w^{1+a_0}$.  Since
	\begin{equation*}
		\big(w^{1+a_0}\big)^{\theta_0}=w,
	\end{equation*}
	we obtain
	\begin{equation}\label{eq:lemma45-interpolated-weighted-square-function}
		\sum_\nu\|G_\nu u\|_{L^2(w)}^2
		\lesssim_{[w^{1+a_0}]_{A_2}}
		2^{s\gamma a_*}
		\|u\|_{L^2(w)}^2,
	\end{equation}
	where
	\begin{align*}
		a_*
		& \coloneqq (d-2)(1-\theta_0)+(3d-1)\theta_0\\
		&=d-2+\frac{2d+1}{4d+5}
		<d-\frac32.
	\end{align*}
	
	We now pass to the commutators.  By Lemma
	\ref{lem:bmo-exponential-weight}, there exists $c_d>0$ such that, whenever
	\begin{equation*}
		|z| \leq \rho
		\coloneqq 
		\frac{c_d}{\|b\|_{\operatorname{BMO}(\mathbb R^d)}},
	\end{equation*}
	the weight
	\begin{equation*}
		w_z(x) \coloneqq e^{2\operatorname{Re}(z)b(x)}
	\end{equation*}
	satisfies $w_z^{1+a_0} \in A_2$ with a uniform bound, namely,
	\begin{equation}\label{eq:lemma45-exponential-weight}
		[w_z^{1+a_0}]_{A_2}\lesssim 1.
	\end{equation}
	For each $\nu$, define the analytic family
	\begin{equation*}
		F_\nu(z)  \coloneqq e^{zb}G_\nu (e^{-zb}u).
	\end{equation*}
	Then
	\begin{equation*}
		\frac{d}{dz}F_\nu(z) \bigg|_{z=0} = [b,G_\nu ]u.
	\end{equation*}
	By Cauchy's integral formula,
	\begin{equation*}
		[b,G_\nu]u
		=\frac{1}{2\pi i}
		\int_{|z|=\rho}
		\frac{F_\nu(z)}{z^2}\,dz.
	\end{equation*}
	Applying the Cauchy-Schwarz inequality to the contour integral and summing over $\nu$ yields
	\begin{align*}
		\sum_\nu\|[b,G_\nu ]u\|_2^2
		&\lesssim
		\rho^{-3}
		\int_{|z|=\rho}
		\sum_\nu\|F_\nu(z)\|_2^2\,|dz|\\
		&=
		\rho^{-3}
		\int_{|z|=\rho}
		\sum_\nu
		\|G_\nu (e^{-zb}u)\|_{L^2(w_z)}^2\,|dz|.
	\end{align*}
	By \eqref{eq:lemma45-interpolated-weighted-square-function} and
	\eqref{eq:lemma45-exponential-weight},
	\begin{equation*}
		\sum_\nu
		\|G_\nu (e^{-zb}u)\|_{L^2(w_z)}^2
		\lesssim
		2^{s\gamma(d-\frac32)}
		\|e^{-zb}u\|_{L^2(w_z)}^2.
	\end{equation*}
	Moreover,
	\begin{equation*}
		\|e^{-zb}u\|_{L^2(w_z)}^2
		=
		\int_{\mathbb R^d}
		|u(x)|^2e^{-2\operatorname{Re}(z)b(x)}
		e^{2\operatorname{Re}(z)b(x)}dx
		=\|u\|_2^2.
	\end{equation*}
	Consequently,
	\begin{align*}
		\sum_\nu\|[b,G_\nu ]u\|_2^2
		&\lesssim
		\rho^{-2}
		2^{s\gamma(d-\frac32)}\|u\|_2^2\\
		&\lesssim
		2^{s\gamma(d-\frac32)}
		\|b\|_{\operatorname{BMO}}^2
		\|u\|_2^2.
	\end{align*}
	This proves the lemma.
	
\end{proof}

We return to $\Esn$.  Let $u\in L^2(\Rd)$ with $\norm u_2=1$.
Cauchy-Schwarz in $\nu$ and Lemma \ref{lem:lemma45-commutator-micro} give
\begin{align*}
	\big|\langle\mathcal E_{s,n,\iota},u\rangle\big|
	&=
	\bigg|
	\sum_\nu
	\langle u_\nu,-[b,G_\nu ]u\rangle
	\bigg|
	\notag\\
	&\leq
	\Big(
	\sum_\nu\|u_\nu\|_2^2
	\Big)^{1/2}
	\Big(
	\sum_\nu\|[b,G_\nu ]u\|_2^2
	\Big)^{1/2}
	\notag\\
	&\lesssim
	2^{\frac12s\gamma(d-\frac32)}
	\|b\|_{\operatorname{BMO}}
	\Big(
	\sum_\nu\|u_\nu\|_2^2
	\Big)^{1/2}.
\end{align*}
Hence,
\begin{equation}\label{eq:lemma45-commutator-reduction}
	\|\mathcal E_{s,n,\iota}\|_2^2
	\lesssim
	2^{s\gamma(d-\frac32)}
	\|b\|_{\operatorname{BMO}}^2
	\sum_\nu\|u_\nu\|_2^2.
\end{equation}
The estimate \cite[(4.29)]{Lai2025} gives, uniformly in $\nu$,
\begin{equation*}
	\|u_\nu\|_{L^2(\mathbb R^d)}^2 \lesssim 2^{-2s\gamma(d-1)} \|\Omega\|_{L^\infty(\mathbb S^{d-1})}^2 \sum_{J\in\mathcal I_{s,\iota}^{\#,n}}
	\sum_{Q\in\mathfrak Q_{J-s}}
	\|h_Q\|_1,
\end{equation*}
Using also $\card \{e_\nu\} \lesssim 2^{s\gamma(d-1)}$, we obtain

\begin{equation}\label{eq:lemma45-all-directions-u}
	\sum_\nu\|u_\nu\|_2^2
	\lesssim
	2^{-s\gamma(d-1)}
	\|\Omega\|_{L^\infty(\mathbb S^{d-1})}^2
	\sum_{J\in\mathcal I_{s,\iota}^{\#,n}}
	\sum_{Q\in\mathfrak Q_{J-s}}
	\|h_Q\|_1.
\end{equation}
Substituting this into \eqref{eq:lemma45-commutator-reduction} gives
\begin{align*}
	\|\mathcal E_{s,n,\iota}\|_2^2
	\lesssim{}&
	2^{-s\gamma/2}
	\|\Omega\|_{L^\infty(\mathbb S^{d-1})}^2
	\|b\|_{\operatorname{BMO}(\mathbb R^d)}^2
	\notag\\
	&\quad\times
	\sum_{J\in\mathcal I_{s,\iota}^{\#,n}}
	\sum_{Q\in\mathfrak Q_{J-s}}
	\|h_Q\|_{L^1(\mathbb R^d)}.
\end{align*}
By Lemma \ref{lem:layered-bad-mass} and $\|\Omega\|_{L^\infty(\mathbb S^{d-1})} \leq 2^{\eta s}$, we obtain
\begin{equation*}
	\|\mathcal{E}_{s,n,\iota}\|_{L^2(\mathbb R^d)}
	\lesssim 2^{-n}2^{-s(\frac{\gamma}{4} - \eta)} \|b\|_{\operatorname{BMO}(\mathbb R^d)}
	\|f\|_{L^1(\mathbb R^d)}^{1/2}.
\end{equation*}
Taking $0<\eta<\gamma/4$ proves Lemma
\ref{lem:sec42-commutator-microlocal}.

\end{proof}

\section{Proof of Lemma \ref{lem:sec42-remainder-microlocal} and Theorem \ref{thm:main}}
\label{sec:proof-lemma47}

\begin{proof}[Proof of Lemma \ref{lem:sec42-remainder-microlocal} ]
We prove Lemma \ref{lem:sec42-remainder-microlocal} by interpolating a rough
$L^3$ bound with an $L^1$ bound having quantitative decay.  As in the preceding
section, set
\begin{equation*}
	u_{J,Q,\nu}
	\coloneqq 
	(I-P_{j-s\kappa})
	\Big[
	\varepsilon_J(T_{j,\nu} h_Q)
	\widetilde\chi_{J^\iota}
	\Big],
\end{equation*}
and let
\begin{equation*}
	\mathcal E_{s,n,\iota}
	\coloneqq 
	\sum_{J\in\mathcal I_{s,\iota}^{\#,n}}
	\sum_{Q\in\mathfrak Q_{J-s}}
	\sum_\nu
	(b-b_Q)(I-G_\nu)u_{J,Q,\nu}.
\end{equation*}
In what follows, $\sum_{J,Q}$ denotes
$\sum_{J\in\mathcal I_{s,\iota}^{\#,n}}
\sum_{Q\in\mathfrak Q_{J-s}}$. 

\noindent{\bf $L^3$ estimate.}
Set
\begin{equation*}
	N \coloneqq d.
\end{equation*}
We claim that
\begin{equation}
	\begin{aligned}
		\norm{\Esn}_{L^3(\Rd)} \lesssim
		2^{s[\gamma(N+d-1)+\eta]}
		2^{-2n/3}
		\norm{b}_{\BMO(\Rd)}
		\norm{f}_{L^1(\Rd)}^{1/3}.
	\end{aligned}
	\label{eq:lemma46-L3}
\end{equation}
Using the identity
\begin{equation*}
	(b-b_Q)(I-G_\nu)u
	=
	(I-G_\nu)\big[(b-b_Q)u\big]
	-[b,G_\nu]u,
\end{equation*}
we decompose
\begin{equation}
	\mathcal E_{s,n,\iota}
	=
	\mathcal E_{s,n,\iota}^{(1)}
	-\mathcal E_{s,n,\iota}^{(2)},
	\label{eq:lemma46-E-splitting}
\end{equation}
where
\begin{equation*}
	\begin{aligned}
		\mathcal E_{s,n,\iota}^{(1)}
		& \coloneqq {}
		\sum_\nu
		(I-G_\nu)
		\left[
		\sum_{J\in\mathcal I_{s,\iota}^{\#,n}}
		\sum_{Q\in\mathfrak Q_{J-s}}
		(b-b_Q)u_{J,Q,\nu}
		\right],
		\\
		\mathcal E_{s,n,\iota}^{(2)}
		& \coloneqq {}
		\sum_\nu
		[b,G_\nu]
		\left[
		\sum_{J\in\mathcal I_{s,\iota}^{\#,n}}
		\sum_{Q\in\mathfrak Q_{J-s}}
		u_{J,Q,\nu}
		\right].
	\end{aligned}
\end{equation*}
We estimate these terms separately.

For the first term, a direct calculation and the support properties give
\begin{equation*}
	|u_{J,Q,\nu}(x)|
	\lesssim
	2^{-jd}\norm{\Omega}_{L^\infty(\Sd)}
	\norm{h_Q}_{L^1(\Rd)}\chi_{4J}(x).
\end{equation*}
Arguing as in the $L^3$ estimates of Lemmas \ref{lem:sec42-boundary-error} and \ref{lem:sec42-low-frequency}, we obtain
\begin{equation}
	\begin{aligned}
		\norm{
			\sum_{J,Q}(b-b_Q)u_{J,Q,\nu}
		}_{L^3(\Rd)}
		\lesssim{}&
		s2^{-2n/3}
		\norm{\Omega}_{L^\infty(\Sd)}
		\\
		&\times
		\norm{b}_{\BMO(\Rd)}
		\norm{f}_{L^1(\Rd)}^{1/3}.
	\end{aligned}
	\label{eq:lemma46-bU-final}
\end{equation}
The multiplier of $I-G_\nu$ satisfies the same Mikhlin condition; hence
\begin{equation*}
	\norm{(I-G_\nu)g}_{L^3(\Rd)}
	\lesssim
	2^{s\gamma N}\norm{g}_{L^3(\Rd)}.
\end{equation*}
Combined with \eqref{eq:lemma46-bU-final}, this gives
\begin{equation}
	\begin{aligned}
		\norm{\mathcal E_{s,n,\iota}^{(1)}}_{L^3(\Rd)}
		&\leq
		\sum_\nu
		\norm{
			(I-G_\nu)
			\left[
			\sum_{J,Q}(b-b_Q)u_{J,Q,\nu}
			\right]
		}_{L^3(\Rd)}
		\\
		&\lesssim
		(1+s)2^{s\gamma(N+d-1)}
		\norm{\Omega}_{L^\infty(\Sd)}
		2^{-2n/3}
		\\
		&\qquad\times
		\norm{b}_{\BMO(\Rd)}
		\norm{f}_{L^1(\Rd)}^{1/3}.
	\end{aligned}
	\label{eq:lemma46-Ea-final}
\end{equation}

For the second term, we first establish the commutator bound
\begin{equation}
	\norm{[b,G_\nu]g}_{L^3(\Rd)}
	\lesssim
	2^{s\gamma N}
	\norm{b}_{\BMO(\Rd)}
	\norm{g}_{L^3(\Rd)}.
	\label{eq:lemma46-commutator-L3}
\end{equation}

Indeed, Lemma \ref{lem:bmo-exponential-weight} allows us to choose
	\begin{equation*}
		\rho
		\coloneqq 
		\frac{c_d}{\norm{b}_{\BMO}},
	\end{equation*}
	so that, whenever $|z|\leq\rho$,
	\begin{equation*}
		w_z(x) \coloneqq e^{3\operatorname{Re}(z)b(x)}
	\end{equation*}
	satisfies
	\begin{equation}
		[w_z]_{A_3}\lesssim1.
		\label{eq:lemma46-exponential-A3}
	\end{equation}
	
	Define the analytic family
	\begin{equation*}
		F_\nu(z)
		\coloneqq 
		e^{zb}G_\nu(e^{-zb}g).
	\end{equation*}
	The Fourier multiplier of $G_\nu$ is
	\begin{equation*}
		m_\nu(\xi)
		\coloneqq 
		\Psi\left(
		2^{s\gamma}
		\left\langle e_\nu,\frac{\xi}{|\xi|}\right\rangle
		\right),
	\end{equation*}
	By the chain rule and the homogeneity of degree zero of $\xi\mapsto\xi/|\xi|$,
	for every $|\alpha|\leq N$ we have
	\begin{equation*}
		|\partial_\xi^\alpha m_\nu(\xi)|
		\lesssim_{d,\alpha}
		2^{s\gamma|\alpha|}|\xi|^{-|\alpha|}
		\lesssim
		2^{s\gamma N}|\xi|^{-|\alpha|}.
	\end{equation*}
	The weighted Mikhlin multiplier theorem and \eqref{eq:lemma46-exponential-A3} yield
	\begin{equation*}
		\begin{aligned}
			\norm{F_\nu(z)}_{L^3(\Rd)}
			&=
			\norm{G_\nu(e^{-zb}g)}_{L^3(w_z)}
			\\
			&\lesssim
			2^{s\gamma N}
			\norm{e^{-zb}g}_{L^3(w_z)}
			\\
			&=
			2^{s\gamma N}
			\norm{g}_{L^3(\Rd)}.
		\end{aligned}
	\end{equation*}
	Finally, Cauchy's integral formula and Minkowski's inequality give
	\begin{equation*}
		\begin{aligned}
			\norm{[b,G_\nu]g}_{L^3(\Rd)}
			&\leq
			\frac1{2\pi}
			\int_{|z|=\rho}
			\frac{\norm{F_\nu(z)}_{L^3(\Rd)}}{|z|^2}|d z|
			\\
			&\leq
			\rho^{-1}
			\sup_{|z|=\rho}
			\norm{F_\nu(z)}_{L^3(\Rd)}
			\\
			&\lesssim
			2^{s\gamma N}
			\norm{b}_{\BMO(\Rd)}
			\norm{g}_{L^3(\Rd)}.
		\end{aligned}
	\end{equation*}
This proves \eqref{eq:lemma46-commutator-L3}.

As in the $L^3$ estimates of Lemmas \ref{lem:sec42-boundary-error} and \ref{lem:sec42-low-frequency}, we have
\begin{equation}
	\norm{
		\sum_{J\in\mathcal I_{s,\iota}^{\#,n}}
		\sum_{Q\in\mathfrak Q_{J-s}}u_{J,Q,\nu}
	}_{L^3(\Rd)}
	\lesssim s
	2^{-2n/3}\norm{\Omega}_{L^\infty(\Sd)}
	\norm{f}_{L^1(\Rd)}^{1/3},
	\label{eq:lemma46-U-final}
\end{equation}

By
\eqref{eq:lemma46-commutator-L3} and
\eqref{eq:lemma46-U-final},
\begin{equation}
	\begin{aligned}
		\norm{\mathcal E_{s,n,\iota}^{(2)}}_{L^3(\Rd)}
		&\leq
		\sum_\nu
		\norm{
			[b,G_\nu]
			\left[
			\sum_{J,Q}u_{J,Q,\nu}
			\right]
		}_{L^3(\Rd)}
		\\
		&\lesssim
		2^{s\gamma N}\norm{b}_{\BMO(\Rd)}
		\sum_\nu
		\norm{
			\sum_{J,Q}u_{J,Q,\nu}
		}_{L^3(\Rd)}
		\\
		&\lesssim
		2^{s\gamma N}
		\card(\Theta_s)
		\norm{\Omega}_{L^\infty(\Sd)}
		2^{-2n/3}
		\norm{b}_{\BMO(\Rd)}
		\norm{f}_{L^1(\Rd)}^{1/3},
	\end{aligned}
	\label{eq:lemma46-Eb-before-direction-count}
\end{equation}
Here $\sum_{J,Q}$ has the meaning fixed above.  Since
$\card(\Theta_s)\lesssim2^{s\gamma(d-1)}$ and
$\norm{\Omega}_{L^\infty}\leq2^{s\eta}$, we obtain
\begin{equation}
	\begin{aligned}
		\norm{\mathcal E_{s,n,\iota}^{(2)}}_{L^3(\Rd)}
		\lesssim &
		2^{s\gamma(N+d-1)}
		\norm{\Omega}_{L^\infty(\Sd)}
		2^{-2n/3}
		\\
		&\times
		\norm{b}_{\BMO(\Rd)}
		\norm{f}_{L^1(\Rd)}^{1/3}.
	\end{aligned}
	\label{eq:lemma46-Eb-final}
\end{equation}

Combining \eqref{eq:lemma46-E-splitting}, \eqref{eq:lemma46-Ea-final}, and \eqref{eq:lemma46-Eb-final} gives
\begin{equation*}
	\begin{aligned}
		\norm{\mathcal E_{s,n,\iota}}_{L^3(\Rd)}
		\lesssim{}&
		s2^{s\gamma(N+d-1)}
		\norm{\Omega}_{L^\infty(\Sd)}
		2^{-2n/3}
		\\
		&\times
		\norm{b}_{\BMO(\Rd)}
		\norm{f}_{L^1(\Rd)}^{1/3}.
	\end{aligned}
\end{equation*}
Using $\norm{\Omega}_{L^\infty(\Sd)}\lesssim2^{s\eta}$, we further obtain
\begin{equation*}
	\begin{aligned}
		\norm{\mathcal E_{s,n,\iota}}_{L^3(\Rd)}
		\lesssim{}&
		2^{s[\gamma(N+d-1)+\eta]}
		2^{-2n/3}
		\norm{b}_{\BMO(\Rd)}
		\norm{f}_{L^1(\Rd)}^{1/3}.
	\end{aligned}
\end{equation*}
This is \eqref{eq:lemma46-L3}.
\vspace{0.3cm}

\noindent{\bf $L^1$ estimate.}
Fix an auxiliary parameter $\varepsilon$ satisfying
\begin{equation*}
	0<\kappa<\varepsilon<1.
\end{equation*}
Choose a radial Littlewood-Paley cutoff
$\psi\in C_c^\infty(\mathbb R^d\setminus\{0\})$ satisfying
\begin{equation*}
	\begin{aligned}
		& \operatorname{supp}\psi
		\subset
		\left\{\xi:\frac12\leq|\xi|\leq2\right\}, \\
		& \quad \sum_{k\in\mathbb Z}\psi(2^k\xi)^2=1,
	\end{aligned}
\end{equation*}
Define the corresponding convolution operators by
\begin{equation*}
	\begin{aligned}
		\widehat{\Lambda_kg}(\xi) &  \coloneqq 
		\psi_k(\xi)\widehat g(\xi), \\
		\psi_k(\xi) &  \coloneqq \psi(2^k\xi).
	\end{aligned}
\end{equation*}
Since $I=\sum_k\Lambda_k^2$, we have
\begin{equation*}
	\begin{aligned}
		(I-G_\nu)(I-P_{j-s\kappa}) = \sum_{k \in \Z} (I-G_\nu) \Lambda_k^2 (I-P_{j-s\kappa})
	\end{aligned}
\end{equation*}

Define the composite multiplier
\begin{equation*}
	\begin{aligned}
		M_{j,k,\nu}(\xi)
		\coloneqq  & \left[
		1- \Psi\left(
		2^{s\gamma}
		\left\langle e_\nu,\frac{\xi}{\abs\xi}\right\rangle
		\right)
		\right] 
		\psi_k(\xi)^2
		\\
		&\times \left[1-\widehat\varpi\big(2^{j-s\kappa}\xi\big)\right]
		,
	\end{aligned}
\end{equation*}
This is the multiplier of
$(I-G_\nu)\Lambda_k^2(I-P_{j-s\kappa})$.  We require the following derivative
bound.
\begin{lemma}
	\label{lem:sec6-M-derivative-bound}
	For every $|\beta| \leq 2N$,
	\begin{equation*}
		\abs{\partial_\xi^\beta M_{j,k,\nu}(\xi)}
		\lesssim \min\{1,2^{j-s\kappa-k}\}
		2^{(s\gamma+k)\abs\beta}.
	\end{equation*}
\end{lemma}

\begin{proof}
	Direct differentiation gives
	\begin{equation*}
		\abs{\partial_\xi^\beta\psi_k(\xi)}
		\lesssim2^{k\abs\beta}.
	\end{equation*}
	The derivatives of the directional cutoff multiplier satisfy
	\begin{equation*}
		\left|
		\partial_\xi^\beta
		\Psi\left(
		2^{s\gamma}
		\left\langle e_\nu,\frac{\xi}{\abs\xi}\right\rangle
		\right)
		\right|
		\lesssim
		2^{(s\gamma+k)\abs\beta},
	\end{equation*}
	on the region $\abs\xi \approx 2^{-k}$. For the factor arising from $I-P_{j-s\kappa}$, when $\beta = 0$ we have 
	\begin{equation*}
		\| 1 - \widehat{\varpi}\|_{L^\infty} \leq 2.
	\end{equation*}
	Moreover, the mean value theorem gives
	\begin{equation*}
		|1 - \widehat{\varpi}(2^{j-s\kappa} \xi)| \leq 2^{j-s\kappa-k} \| \partial \widehat{\varpi} \|_{L^\infty}.
	\end{equation*}
	When $|\beta| \geq 1$, the Schwartz decay of $\widehat\varpi$ gives 
	\begin{equation*}
		\begin{aligned}
			\left|
			\partial_\xi^\beta
			\widehat\varpi\big(2^{j-s\kappa}\xi\big)
			\right| & \lesssim 2^{(j-s\kappa)\abs\beta}
			\left(1+2^{j-s\kappa}\abs\xi\right)^{-L} \\
			& \lesssim 2^{(j-s\kappa)(|\beta|-L)} |\xi|^{-L}.
		\end{aligned}
	\end{equation*}
	Applying these estimates with $L=|\beta|$ and $L=|\beta|-1$, as
	appropriate, and using $|\xi|\approx2^{-k}$ yields
	\begin{equation*}
		\left|
		\partial_\xi^\beta
		\widehat\varpi\big(2^{j-s\kappa}\xi\big)
		\right|
		\lesssim \min\{1,2^{j-s\kappa-k}\} 2^{k|\beta|}.
	\end{equation*}
	Combining these estimates with the product rule proves the lemma.	
\end{proof}

Let $D_{j,k,\nu}(x,y)$ denote the kernel obtained by applying
$(I-G_\nu) \Lambda_k^2 \big(I-P_{j-s\kappa}\big)$
to $\big(T_{j,\nu} h_Q\big)\widetilde\chi_{J^\iota}$; thus
\begin{equation}
	\begin{aligned}
		& (I-G_\nu) 	\Lambda_k^2 \big(I-P_{j-s\kappa}\big)
		\left[
		\big(T_{j,\nu} h_Q\big)
		\widetilde\chi_{J^\iota}
		\right](x)
		\\
		&\qquad=:
		\int_Q D_{j,k,\nu}(x,y)h_Q(y)\,\dd y,
	\end{aligned}
	\label{eq:lemma45-high-kernel-action}
\end{equation}
where
\begin{equation*}
	D_{j,k,\nu}(x,y)
	=C\int_{\Rd}e^{i\langle x-w,\xi\rangle}M_{j,k,\nu}(\xi)
	\int_{\Rd}\mathcal K_{j,\nu}(w-y)
	\widetilde\chi_{J^\iota}(w)\,\dd w\,\dd\xi.
\end{equation*}
Making the change of variables $w-y=r\theta$ with $\theta\in E_\nu$, we obtain
\begin{equation}
	\begin{aligned}
		D_{j,k,\nu}(x,y)
		=C
		\int_{E_\nu}\Omega(\theta)
		\int_{\Rd}\int_0^\infty
		&e^{ i\langle x-y-r\theta,\xi\rangle}
		M_{j,k,\nu}(\xi)
		\\
		&\times
		\frac{\varphi_j(r)}{r}
		\widetilde\chi_{J^\iota}(y+r\theta)
		\,\dd r\,\dd\xi\,\dd\sigma(\theta).
	\end{aligned}
	\label{eq:lemma46-D}
\end{equation}

We shall also use the following weighted spatial-kernel estimate.
\begin{lemma}
	\label{lem:sec6-BMO-kernel}
	Let $y\in Q$ and $r\approx2^j$.  Suppose that either $z=y+r\theta$ or
	$z=y_t+r\theta$, where
	\begin{equation*}
		y_t=ty+(1-t)y_Q,
		\quad 0\leq t\leq1.
	\end{equation*}
	Then
	\begin{equation*}
		\int_{\Rd}
		\frac{\abs{b(x)-b_Q}}
		{\left(1+2^{-2k}\abs{x-z}^{2}\right)^N}
		\,\dd x
		\lesssim
		2^{kd}\big(s+\abs{j-k}\big)\norm{b}_{\BMO}.
	\end{equation*}	
\end{lemma}

\begin{proof}
	Let $R=R_k(z)$ be the cube centered at $z$ with side length comparable to
	$2^k$.
	Choose a cube $S$ such that
	\begin{equation*}
		Q\cup R\subset S,
		\qquad
		\ell(S)\approx2^{\max\{j,k\}}.
	\end{equation*}
	Since $\ell(Q)=2^{j-s}$, we have
	\begin{equation*}
		\begin{aligned}
			\abs{b_Q-b_R}
			&\leq
			\abs{b_Q-b_S}+\abs{b_S-b_R}
			\\
			&\lesssim
			\big(s+\abs{j-k}\big)\norm{b}_{\BMO}.
		\end{aligned}
	\end{equation*}
	A dyadic annular decomposition centered at $R$, together with $2N>d$, gives
	\begin{equation*}
		\begin{aligned}
			&\int_{\Rd}
			\frac{\abs{b(x)-b_R}}
			{\left(1+2^{-2k}\abs{x-z}^{2}\right)^N}
			\,\dd x
			\\
			&\quad\lesssim
			\sum_{m=0}^\infty
			2^{-2Nm}
			\int_{2^{m+1}R}\abs{b(x)-b_R}\,\dd x
			\\
			&\quad\lesssim
			\sum_{m=0}^\infty
			2^{-2Nm}(m+1)2^{md}\abs R\norm{b}_{\BMO}
			\\
			&\quad\lesssim
			2^{kd}\norm{b}_{\BMO}.
		\end{aligned}
	\end{equation*}
	Combining this with
	\begin{equation*}
		\int_{\Rd}
		\left(1+2^{-2k}\abs{x-z}^{2}\right)^{-N}\,\dd x
		\lesssim2^{kd}
	\end{equation*}
	and
	$\abs{b-b_Q}\leq\abs{b-b_R}+\abs{b_R-b_Q}$,
	proves the lemma.
\end{proof}

We first consider the high-frequency contribution
\begin{equation}
	\begin{aligned}
		\mathcal E_{s,n,\iota}^{\mathrm{high}}
		\coloneqq  \sum_{J,Q} \sum_\nu
		\sum_{k\leq j-s\varepsilon}
		&(b-b_Q)
		\big(I-P_{j-s\kappa}\big)
		\Lambda_k^2
		(I-G_\nu) \big[\varepsilon_J(T_{j,\nu} h_Q)\widetilde\chi_{J^\iota} \big].
	\end{aligned}
	\label{eq:lemma45-high-definition}
\end{equation}
We combine the three Fourier multipliers and estimate the kernel of their
composition directly.  Minkowski's inequality gives
\begin{equation*}
	\begin{aligned}
		\norm{\mathcal E_{s,n,\iota}^{\mathrm{high}}}_{L^1(\Rd)}
		\leq
		\sum_{J,Q} \sum_\nu
		\sum_{k\leq j-s\varepsilon}
		\norm{
			(b-b_Q)		\big(I-G_\nu\big) 	\Lambda_k^2 \big(I-P_{j-s\kappa}\big) \big[ (T_{j,\nu} h_Q)\widetilde\chi_{J^\iota} \big]
		}_{L^1(\Rd)}.
	\end{aligned}
\end{equation*}

Let $N_1$ be a sufficiently large positive integer.  Integrating by parts
$N_1$ times in $r$ in \eqref{eq:lemma46-D} gives
\begin{equation*}
	\begin{aligned}
		&\int_0^\infty
		e^{- i r\langle\theta,\xi\rangle}
		\frac{\varphi_j(r)}{r}
		\widetilde\chi_{J^\iota}(y+r\theta)\dd r
		\\
		&\quad=
		\big( i\langle\theta,\xi\rangle\big)^{-N_1}
		\int_0^\infty
		e^{- i r\langle\theta,\xi\rangle}
		\partial_r^{N_1}
		\left[
		\frac{\varphi_j(r)}{r}
		\widetilde\chi_{J^\iota}(y+r\theta)
		\right]\dd r.
	\end{aligned}
\end{equation*}

The Leibniz rule gives
\begin{equation}
	\begin{aligned}
		&\partial_r^{N_1}
		\left[
		\frac{\varphi_j(r)}{r}
		\widetilde\chi_{J^\iota}(y+r\theta)
		\right]
		\\
		&\quad=
		\sum_{a=0}^{N_1}
		\binom{N_1}{a}
		\partial_r^a
		\left[\widetilde\chi_{J^\iota}(y+r\theta)\right]
		\partial_r^{N_1-a}
		\left[\frac{\varphi_j(r)}{r}\right].
	\end{aligned}
	\label{eq:lemma45-high-Leibniz-r}
\end{equation}
Since $r \approx 2^j$, the derivative bounds for the smooth cutoff give 
\begin{equation*}
	\left|
	\partial_r^a
	\left[\widetilde\chi_{J^\iota}(y+r\theta)\right]
	\right|
	\lesssim_a
	2^{-a(j-s\kappa)},
\end{equation*}
and
\begin{equation*}
	\left|
	\partial_r^{N_1-a}
	\left[\frac{\varphi_j(r)}{r}\right]
	\right|
	\lesssim_{N_1}
	2^{-j(N_1-a+1)}.
\end{equation*}
Thus \eqref{eq:lemma45-high-Leibniz-r} yields
\begin{equation}
	\begin{aligned}
		\left|
		\partial_r^{N_1}
		\left[
		\frac{\varphi_j(r)}{r}
		\widetilde\chi_{J^\iota}(y+r\theta)
		\right]
		\right|
		&\lesssim_{N_1}
		\sum_{a=0}^{N_1}
		2^{-a(j-s\kappa)}
		2^{-j(N_1-a+1)}
		\\
		&\lesssim_{N_1}
		2^{s\kappa N_1}2^{-j(N_1+1)}.
	\end{aligned}
	\label{eq:lemma45-high-r-derivative}
\end{equation}

The support of $\psi_k$ implies that $|\xi|\approx2^{-k}$.  Integrating by
parts in $\xi$ gives
\begin{equation*}
	\begin{aligned}
		& \quad \int_{|\xi| \approx 2^{-k}} e^{i \inner{x-y-r\theta}{\xi}} \big( i \inner{\theta}{\xi} \big)^{-N_1} M_{j,k,\nu}(\xi)\,\dd \xi \\
		& = \int_{|\xi| \approx 2^{-k}} e^{i \inner{x-y-r\theta}{\xi}} \frac{\big( I - 2^{-2k} \Delta_\xi\big)^N \big[ \big( i \inner{\theta}{\xi} \big)^{-N_1} M_{j,k,\nu}(\xi) \big] }{\big( 1+2^{-2k} |x-y-r\theta|^2 \big)^N }\,\dd\xi,
	\end{aligned}
\end{equation*}
where $N = d$.

By the support properties of $1-\Psi$, 
\begin{equation*}
	\left|
	\left\langle e_\nu,\frac{\xi}{\abs\xi}\right\rangle
	\right|
	\geq 2^{1-s\gamma}.
\end{equation*}
Since $\theta\in E_\nu$ implies
$\abs{\theta-e_\nu}\leq2^{-s\gamma-2}$, we obtain
\begin{equation*}
	\left|
	\left\langle\theta,\frac{\xi}{\abs\xi}\right\rangle
	\right|
	\geq
	\left|
	\left\langle e_\nu,\frac{\xi}{\abs\xi}\right\rangle
	\right|
	-\abs{\theta-e_\nu}
	\gtrsim2^{-s\gamma}.
\end{equation*}
The support of $\psi_k$ gives $|\xi|\approx2^{-k}$.  Hence
\begin{equation*}
	\abs{\langle\theta,\xi\rangle}
	\gtrsim2^{-s\gamma-k}.
\end{equation*}
Furthermore, for every multi-index $|\alpha| \leq 2N$, 
\begin{equation*}
	\left|
	\partial_\xi^\alpha
	\left(\langle\theta,\xi\rangle^{-N_1}\right)
	\right|
	\lesssim
	2^{(s\gamma+k)(N_1+\abs\alpha)}.
\end{equation*}
Together with Lemma \ref{lem:sec6-M-derivative-bound}, this yields
\begin{equation}
	\begin{aligned}
		&\left|
		\big(I-2^{-2k}\Delta_\xi\big)^N
		\left[
		\langle\theta,\xi\rangle^{-N_1}
		M_{j,k,\nu}(\xi)
		\right]
		\right| \lesssim 2^{(s\gamma+k)N_1}2^{2Ns\gamma}.
	\end{aligned}
	\label{eq:lemma45-high-xi-derivative}
\end{equation}

Substituting \eqref{eq:lemma45-high-r-derivative} and
\eqref{eq:lemma45-high-xi-derivative} into
\eqref{eq:lemma46-D}, we obtain the pointwise kernel bound
\begin{equation}
	\begin{aligned}
		\abs{D_{j,k,\nu}(x,y)}
		\lesssim &
		2^{-j(N_1+1)}
		2^{s\kappa N_1}
		2^{(s\gamma+k)N_1}
		2^{2Ns\gamma}
		\\
		&\times
		\int_{E_\nu}\abs{\Omega(\theta)}
		\int_{r \approx 2^j}
		\int_{\abs\xi \approx 2^{-k}}
		\frac{\dd\xi\dd r\dd\sigma(\theta)}
		{\left(
			1+2^{-2k}\abs{x-y-r\theta}^{2}
			\right)^N}.
	\end{aligned}
	\label{eq:lemma45-high-pointwise-kernel}
\end{equation}

By \eqref{eq:lemma45-high-kernel-action} and Fubini's theorem,
\begin{equation*}
	\begin{aligned}
		&\norm{
			(b-b_Q) (I-G_\nu)  \Lambda_k^2  
			(I-P_{j-s\kappa}) \big[(T_{j,\nu} h_Q)\widetilde\chi_{J^\iota} \big]
		}_{L^1(\Rd)}
		\\
		&\quad\leq
		\int_Q\abs{h_Q(y)}
		\int_{\Rd}
		\abs{b(x)-b_Q}
		\abs{D_{j,k,\nu}(x,y)}
		\dd x\,\dd y.
	\end{aligned}
\end{equation*}
Insert \eqref{eq:lemma45-high-pointwise-kernel} and apply Lemma \ref{lem:sec6-BMO-kernel}. Since $|\xi| \approx 2^{-k}$, integration in frequency contributes $2^{-kd}$. Similarly, $r \approx 2^j$ means that integration in $r$ contributes $2^j$. Hence,
\begin{equation*}
	\begin{aligned}
		&\norm{
			(b-b_Q)
			\big(I-P_{j-s\kappa}\big)
			\Lambda_k^2(I-G_\nu)\big[(T_{j,\nu} h_Q)\widetilde\chi_{J^\iota} \big]
		}_{L^1(\Rd)}
		\\
		&\quad\lesssim 
		\big(s+j-k\big)
		\norm{b}_{\BMO}
		\norm{h_Q}_{L^1(\Rd)}
		\\
		&\qquad\times
		2^{-(j-k)N_1}
		2^{s\kappa N_1+s\gamma N_1+2Ns\gamma}
		\int_{E_\nu}\abs{\Omega(\theta)}\dd\sigma(\theta).
	\end{aligned}
\end{equation*}
Recall that $\norm{\Omega}_{L^\infty(\Sd)}\leq2^{s\eta}$ and
$\abs{E_\nu}\lesssim2^{-s\gamma(d-1)}$.  It follows that
\begin{equation*}
	\begin{aligned}
		&\norm{
			(b-b_Q)
			\big(I-P_{j-s\kappa}\big)
			\Lambda_k^2(I-G_\nu)\big[(T_{j,\nu} h_Q)\widetilde\chi_{J^\iota} \big]
		}_{L^1(\Rd)}
		\\
		&\quad\lesssim 
		\big(s+j-k\big)
		\norm{b}_{\BMO}
		\norm{h_Q}_{L^1(\Rd)}
		\\
		&\qquad\times
		2^{-s\gamma(d-1)}
		2^{-(j-k)N_1}
		2^{s\kappa N_1+s\gamma N_1+2Ns\gamma+s\eta}.
	\end{aligned}
\end{equation*}
Summing first over $k$ and using the high-frequency condition $k\leq j-s\varepsilon$ gives
\begin{equation*}
	\begin{aligned}
		&\sum_{k\leq j-s\varepsilon}
		\norm{
			(b-b_Q) (I-G_\nu) \Lambda_k^2 \big(I-P_{j-s\kappa}\big) \big[(T_{j,\nu} h_Q)\widetilde\chi_{J^\iota} \big]
		}_{L^1(\Rd)}
		\\
		&\quad\lesssim
		\norm{b}_{\BMO}
		\norm{h_Q}_{L^1(\Rd)}
		2^{-s\gamma(d-1)}2^{-s(N_1(\varepsilon-\kappa-\gamma)
			-2N\gamma-\eta)}.
	\end{aligned}
\end{equation*}
Summation in $\nu$ cancels the factor $2^{-s\gamma(d-1)}$.  Summing next in
$J,Q$ and applying Lemma \ref{lem:layered-bad-mass}, we obtain
\begin{equation}
	\begin{aligned}
		\norm{\mathcal E_{s,n,\iota}^{\mathrm{high}}}_{L^1(\Rd)}
		\lesssim 2^{-2n}2^{-s[N_1(\varepsilon-\kappa-\gamma)
			-2N\gamma-\eta]}
		\norm{b}_{\BMO}
		\norm{f}_{L^1(\Rd)}.
	\end{aligned}
	\label{eq:lemma45-high-L1-before-absorption}
\end{equation}
If $\varepsilon>\kappa+\gamma$, choosing $N_1$ sufficiently large yields the
required decay for the high-frequency contribution.

We next turn to the low-frequency contribution.  Fix the parameters in the
order
\begin{equation*}
	0<\eta\ll\gamma\ll\kappa<\varepsilon\ll1.
\end{equation*}
Define
\begin{equation*}
	\begin{aligned}
		\mathcal E_{s,n,\iota}^{\mathrm{low}}
		\coloneqq  \sum_{J,Q} \sum_\nu
		\sum_{k>j-s\varepsilon}
		&(b-b_Q)
		\big(I-P_{j-s\kappa}\big)
		\Lambda_k^2
		\big(I-G_\nu\big)
		\big[\varepsilon_J (T_{j,\nu} h_Q)\widetilde\chi_{J^\iota} \big].
	\end{aligned}
\end{equation*}
By Minkowski's inequality,
\begin{equation*}
	\begin{aligned}
		\norm{\mathcal E_{s,n,\iota}^{\mathrm{low}}}_{L^1(\Rd)}
		\leq \sum_{J,Q} \sum_\nu
		\sum_{k>j-s\varepsilon}
		\norm{
			(b-b_Q)
			\big(I-P_{j-s\kappa}\big)
			\Lambda_k^2
			\big(I-G_\nu\big)
			\big[(T_{j,\nu} h_Q)\widetilde\chi_{J^\iota} \big]
		}_{L^1(\Rd)}.
	\end{aligned}
\end{equation*}
Integration by parts in $\xi$ in \eqref{eq:lemma46-D} gives
\begin{equation*}
	\begin{aligned}
		D_{j,k,\nu}(x,y)
		=C_d
		\int_{E_\nu}\Omega(\theta)
		\int_0^\infty\frac{\varphi_j(r)}{r}
		\int_{\Rd}
		&e^{ i\langle x-y-r\theta,\xi\rangle}
		\widetilde\chi_{J^\iota}(y+r\theta)
		\\
		&\times
		\frac{M_{j,k,\nu}^{[2N]}(\xi)}
		{\left(1+2^{-2k}\abs{x-y-r\theta}^{2}\right)^N}
		\,\dd\xi\,\dd r\,\dd\sigma(\theta),
	\end{aligned}
\end{equation*}
where
\begin{equation*}
	M_{j,k,\nu}^{[2N]}(\xi)
	\coloneqq 
	(I-2^{-2k} \Delta_\xi)^N
	M_{j,k,\nu}(\xi).
	\label{eq:lemma45-low-L}
\end{equation*}

Let $y_Q$ be the center of $Q$.  Since $h_Q$ has mean zero,
\begin{equation*}
	\int_QD_{j,k,\nu}(x,y)h_Q(y)\,\dd y
	=
	\int_Q
	\left[
	D_{j,k,\nu}(x,y)-D_{j,k,\nu}(x,y_Q)
	\right]
	h_Q(y)\,\dd y.
\end{equation*}
For brevity, write
\begin{equation*}
	W(x-y-r\theta)
	\coloneqq 
	\left(1+2^{-2k}\abs{x-y-r\theta}^{2}\right)^{-N}.
\end{equation*}
Decompose
\begin{equation*}
	D_{j,k,\nu}(x,y)-D_{j,k,\nu}(x,y_Q)
	=
	D_{j,k,\nu}^{(1)}(x,y)
	+D_{j,k,\nu}^{(2)}(x,y)
	+D_{j,k,\nu}^{(3)}(x,y),
\end{equation*}
where the first term records the difference of phases:
\begin{equation*}
	\begin{aligned}
		D_{j,k,\nu}^{(1)}(x,y)
		\coloneqq C_d
		& \int_{E_\nu}\Omega(\theta)
		\int_0^\infty\frac{\varphi_j(r)}{r}
		\int_{\Rd} e^{ i\langle x-r\theta,\xi\rangle}
		\left(
		e^{- i\langle y,\xi\rangle}
		-e^{- i\langle y_Q,\xi\rangle}
		\right)
		\\
		& \quad \times
		\widetilde\chi_{J^\iota}(y+r\theta)
		W(x-y-r\theta)
		M_{j,k,\nu}^{[2N]}(\xi)
		d \xi d r d \sigma(\theta);
	\end{aligned}
\end{equation*}
the second term records the difference of the smooth cutoffs:
\begin{equation*}
	\begin{aligned}
		D_{j,k,\nu}^{(2)}(x,y)
		&  \coloneqq C_d
		\int_{E_\nu}\Omega(\theta)
		\int_0^\infty\frac{\varphi_j(r)}{r}
		\int_{\Rd} e^{ i\langle x-y_Q-r\theta,\xi\rangle}
		\\
		&\qquad \quad \times
		\left[
		\widetilde\chi_{J^\iota}(y+r\theta)
		-\widetilde\chi_{J^\iota}(y_Q+r\theta)
		\right] \\ 
		&  \qquad \quad \times W(x-y-r\theta)
		M_{j,k,\nu}^{[2N]}(\xi)
		d \xi d r d \sigma(\theta);
	\end{aligned}
\end{equation*}
and the third term records the difference in the centers of the denominators:
\begin{equation*}
	\begin{aligned}
		D_{j,k,\nu}^{(3)}(x,y)
		&  \coloneqq C_d
		\int_{E_\nu}\Omega(\theta)
		\int_0^\infty\frac{\varphi_j(r)}{r}
		\int_{\Rd} e^{ i\langle x-y_Q-r\theta,\xi\rangle}
		\widetilde\chi_{J^\iota}(y_Q+r\theta)
		\\
		& \qquad \quad \times
		\left[
		W(x-y-r\theta)-W(x-y_Q-r\theta)
		\right]
		M_{j,k,\nu}^{[2N]}(\xi)
		d \xi d r d \sigma(\theta).
	\end{aligned}
\end{equation*}

For $D_{j,k,\nu}^{(1)}$, the inequalities
$\abs{e^{iu}-e^{iv}}\leq\abs{u-v}$,
$\abs{y-y_Q}\lesssim2^{j-s}$, and $\abs\xi\approx2^{-k}$ give
\begin{equation*}
	\left|
	e^{- i\langle y,\xi\rangle}
	-e^{- i\langle y_Q,\xi\rangle}
	\right|
	\lesssim
	2^{j-s-k}.
\end{equation*}
The size bound for the kernel and the microlocal decomposition give
\begin{equation}
	\int_{E_\nu}\abs{\Omega(\theta)}d \sigma(\theta)
	\lesssim
	2^{-s\gamma(d-1)+s\eta}.
	\label{eq:lemma45-low-angular-mass}
\end{equation}
Combining Lemmas \ref{lem:sec6-M-derivative-bound} and \ref{lem:sec6-BMO-kernel}, we obtain
\begin{equation*}
	\begin{aligned}
		&\int_{\Rd}
		\abs{b(x)-b_Q}
		\abs{D_{j,k,\nu}^{(1)}(x,y)}d x
		\\
		&\quad\lesssim
		\big(s+\abs{j-k}\big)
		\norm{b}_{\BMO}
		2^{j-s\kappa-k}2^{2Ns\gamma}
		2^{-s\gamma(d-1)+s\eta}
		2^{j-s-k}.
	\end{aligned}
\end{equation*}
Here integration over the frequency region contributes $O(2^{-kd})$, while
Lemma \ref{lem:sec6-BMO-kernel} contributes $2^{kd}$;
these factors cancel. Moreover,
\begin{equation*}
	\int_{r\approx2^j}\frac{\abs{\varphi_j(r)}}{r}d r
	\lesssim1.
\end{equation*}

For $D_{j,k,\nu}^{(2)}$, the mean value theorem and
\begin{equation*}
	\norm{\nabla\widetilde\chi_{J^\iota}}_{L^\infty(\Rd)}
	\lesssim2^{-j+s\kappa},
\end{equation*}
we obtain
\begin{equation*}
	\begin{aligned}
		&\left|
		\widetilde\chi_{J^\iota}(y+r\theta)
		-\widetilde\chi_{J^\iota}(y_Q+r\theta)
		\right|
		\\
		&\qquad\leq
		\abs{y-y_Q}
		\norm{\nabla\widetilde\chi_{J^\iota}}_{L^\infty(\Rd)}
		\lesssim
		2^{j-s}2^{-j+s\kappa}
		=2^{-(1-\kappa)s}.
	\end{aligned}
\end{equation*}
As in the estimate of $D_{j,k,\nu}^{(1)}$, Lemmas \ref{lem:sec6-M-derivative-bound} and \ref{lem:sec6-BMO-kernel},
together with
\eqref{eq:lemma45-low-angular-mass}, give
\begin{equation*}
	\begin{aligned}
		&\int_{\Rd}
		\abs{b(x)-b_Q}
		\abs{D_{j,k,\nu}^{(2)}(x,y)}d x
		\\
		&\quad\lesssim
		\big(s+\abs{j-k}\big)
		\norm{b}_{\BMO}
		2^{j-s\kappa-k}2^{2Ns\gamma}
		2^{-s\gamma(d-1)+s\eta}
		2^{-(1-\kappa)s}.
	\end{aligned}
\end{equation*}

For $D_{j,k,\nu}^{(3)}$, direct differentiation gives
\begin{equation*}
	\begin{aligned}
		\left|
		\nabla_z
		\left(1+2^{-2k}\abs{x-z-r\theta}^{2}\right)^{-N}
		\right| & \lesssim_N 
		2^{-k}
		\frac{2^{-k}\abs{x-z-r\theta}}
		{\left(1+2^{-2k}\abs{x-z-r\theta}^{2}\right)^{N+1}} \\
		& \lesssim \frac{2^{-k}}{\left(1+2^{-2k}\abs{x-z-r\theta}^{2}\right)^N}.
	\end{aligned}
\end{equation*}
Set
\begin{equation*}
	y_t \coloneqq ty+(1-t)y_Q,
	\qquad 0\leq t\leq1.
\end{equation*}
By the fundamental theorem of calculus,
\begin{equation*}
	\begin{aligned}
		\abs{W(x-y-r\theta)-W(x-y_Q-r\theta)}
		\lesssim \int_0^1
		\frac{2^{j-s-k}}
		{\left(1+2^{-2k}\abs{x-y_t-r\theta}^{2}\right)^N}
		d t.
	\end{aligned}
\end{equation*}
As for the preceding two terms, Lemmas \ref{lem:sec6-M-derivative-bound} and \ref{lem:sec6-BMO-kernel}, together with \eqref{eq:lemma45-low-angular-mass}, yield
\begin{equation*}
	\begin{aligned}
		&\int_{\Rd}
		\abs{b(x)-b_Q}
		\abs{D_{j,k,\nu}^{(3)}(x,y)}d x
		\\
		&\quad\lesssim
		\big(s+\abs{j-k}\big)
		\norm{b}_{\BMO}
		2^{j-s\kappa-k}2^{2Ns\gamma}
		2^{-s\gamma(d-1)+s\eta}
		2^{j-s-k}.
	\end{aligned}
\end{equation*}

Combining the estimates for $D^{(1)}$-$D^{(3)}$, we obtain
\begin{equation*}
	\begin{aligned}
		&\norm{
			(b-b_Q)
			\big(I-P_{j-s\kappa}\big)
			\Lambda_k^2(I-G_\nu)\big[(T_{j,\nu} h_Q)\widetilde\chi_{J^\iota} \big]
		}_{L^1(\Rd)}
		\\
		&\quad\lesssim
		\big(s+\abs{j-k}\big)
		\norm{b}_{\BMO}
		2^{j-s\kappa-k}2^{2Ns\gamma}
		2^{-s\gamma(d-1)+s\eta}
		\\
		&\qquad\times
		\Big[
		2^{j-s-k}+2^{-(1-\kappa)s}
		\Big]
		\norm{h_Q}_{L^1(\Rd)}.
	\end{aligned}
\end{equation*}
Summing over $k>j-s\varepsilon$ gives
\begin{equation*}
	\begin{aligned}
		&\sum_{k>j-s\varepsilon}
		\norm{
			(b-b_Q)
			\big(I-P_{j-s\kappa}\big)
			\Lambda_k^2(I-G_\nu)\big[(T_{j,\nu} h_Q)\widetilde\chi_{J^\iota} \big]
		}_{L^1(\Rd)}
		\\
		& \quad \lesssim \norm{b}_{\BMO}
		\norm{h_Q}_{L^1(\Rd)}
		2^{-s\gamma(d-1)} \\
		& \qquad \times \left[
		2^{-s(1+\kappa-2\varepsilon-\eta-2N\gamma)}+2^{-s(1-\varepsilon-\eta-2N\gamma)}
		\right],
	\end{aligned}
\end{equation*}
Summation in $\nu$ cancels the factor $2^{-s\gamma(d-1)}$.  Lemma
\ref{lem:layered-bad-mass} then gives
\begin{equation*}
	\begin{aligned}
		& \sum_{J\in\mathcal I_{s,\iota}^{\#,n}}
		\sum_{Q\in\mathfrak Q_{J-s}}
		\sum_\nu \sum_{k > j-s\varepsilon} \norm{
			(b-b_Q)
			\big(I-P_{j-s\kappa}\big)
			\Lambda_k^2(I-G_\nu)\big[(T_{j,\nu} h_Q)\widetilde\chi_{J^\iota} \big]
		}_{L^1(\Rd)} \\
		& \quad \lesssim 2^{-2n}
		\left[
		2^{-s(1+\kappa-2\varepsilon-\eta-2N\gamma)}+2^{-s(1-\varepsilon-\eta-2N\gamma)}
		\right] \norm{b}_{\BMO}
		\norm{f}_{L^1(\Rd)}.
	\end{aligned}
\end{equation*}
Thus it suffices that
\begin{equation*}
	\begin{aligned}
		& 1+\kappa-2\varepsilon-\eta-2N\gamma>0, \\
		& \ 1-\varepsilon-\eta-2N\gamma>0,
	\end{aligned}
\end{equation*}
which gives the required decay for the low-frequency contribution.

Combining the $L^1$ estimates for the high- and low-frequency parts yields
\begin{equation}
	\label{eq:lemma46-L1}
	\begin{aligned}
		\| \Esn \|_{L^1} & \lesssim 2^{-2n} \|b\|_{\BMO} \|f\|_{L^1} \\
		& \quad\times\big[
		2^{-s[N_1(\varepsilon-\kappa-\gamma)-2N\gamma-\eta]}
		+2^{-s(1+\kappa-2\varepsilon-\eta-2N\gamma)}
		+2^{-s(1-\varepsilon-\eta-2N\gamma)}\big].
	\end{aligned}
\end{equation}

\noindent{\bf Interpolation.}
Interpolating \eqref{eq:lemma46-L3} with \eqref{eq:lemma46-L1}, we obtain
\begin{equation*}
	\begin{aligned}
		\|\Esn\|_{L^2} & \lesssim \| \Esn \|_{L^1}^{1/4} \| \Esn \|_{L^3}^{3/4} \\
		& \lesssim 2^{-n} (2^{-s\theta_1} + 2^{-s \theta_2} + 2^{-s \theta_3}) \|b\|_{\BMO} \|f\|_{L^1}^{1/2},
	\end{aligned}
\end{equation*}
where
\begin{align*}
	\theta_1 & = \frac{1}{4} \big[ N_1(\varepsilon-\kappa-\gamma)-2N\gamma-\eta \big] - \frac{3}{4} \big[ \gamma(N+d-1) + \eta \big], \\
	\theta_2 & = \frac{1}{4} \big[ 1+\kappa-2\varepsilon-\eta-2N\gamma \big] - \frac{3}{4} \big[ \gamma(N+d-1) + \eta \big], \\
	\theta_3 & = \frac{1}{4} \big[ 1-\varepsilon-\eta-2N\gamma \big] - \frac{3}{4} \big[ \gamma(N+d-1) + \eta \big]. \\
\end{align*}
Choose $\varepsilon$ and $\kappa$ sufficiently small with
$0<\kappa<\varepsilon$, then choose $\gamma$ and $\eta$ successively smaller,
and finally take $N_1$ sufficiently large.  With these choices,
$\theta_1,\theta_2,\theta_3>0$, proving Lemma
\ref{lem:sec42-remainder-microlocal}.
\end{proof}
\vspace{0.3cm}
\begin{proof}[Proof of  Theorem \ref{thm:main}]
The five estimates in Subsection
\ref{subsec:smoothing-microlocal-decomposition} now prove Proposition
\ref{prop:sec4-reindexed-linearized-estimate}.  Consequently, Proposition
\ref{prop:sec3-linearized-signed-L2}, the Rademacher-Menshov reduction, and
Proposition \ref{prop:small-omega-key-L2} follow in succession.  The reductions
in Section \ref{sec:reduction-linearization} then prove Theorem
\ref{thm:dyadic}, and hence Theorem \ref{thm:main}.
\end{proof}


\end{document}